\newif\ifdraft

\documentclass[12pt]{amsart}

\usepackage[margin=1.35in,marginpar=1in]{geometry}
\usepackage{enumitem}
\usepackage{xcolor}
\usepackage{url}
\usepackage{amsmath}
\usepackage{bbm}
\usepackage{amsthm}
\usepackage{comment}
\usepackage{soul}

\ifdraft
\usepackage{fancyhdr,datetime2}
\fancypagestyle{firstpage}{%
	
	\fancyhf{}
	\fancyhead[C]{\color{red}\fbox{\textbf{DRAFT} (\DTMnow) -- NOT FOR DISTRIBUTION}}
	\fancyfoot[C]{\thepage}
}
\fi

\numberwithin{equation}{section}

\theoremstyle{plain}
\newtheorem*{RPF}{Ruelle--Perron--Frobenius (RPF) Theorem}
\newtheorem{theorem}{Theorem}[section]

\newtheorem{lemma}[theorem]{Lemma}

\theoremstyle{definition}
\newtheorem{definition}[theorem]{Definition}
\newtheorem{example}[theorem]{Example}

\newcommand{\foot}[1]{\marginpar{\raggedright\tiny #1}}
\newcommand{\vc}[1]{\foot{VC: #1}}

\setenumerate{itemsep=1ex}

\subjclass[2020]{Primary: 37D35, 37D20. Secondary: 37D25, 37H05}
\keywords{Ruelle--Perron--Frobenius theorem, convex cones, Hilbert projective metric, Birkhoff contraction theorem, nonstationary dynamics}

\title{A nonstationary Ruelle--Perron--Frobenius theorem}
\author{Vaughn Climenhaga and Gregory Hemenway}
\address{Dept.\ of Mathematics, University of Houston, Houston, TX 77204}
\email{climenha@math.uh.edu}
\address{Dept.\ of Mathematics, The Ohio State University, Columbus, OH 43210}
\email{hemenway.math@gmail.com}

\date{\today}

\thanks{VC: This material is based upon work supported by the National Science Foundation under Award No.\ DMS-1554794, DMS-2154378, and DMS-2453314.
GH: This material is based upon work supported by the National Science Foundation under Award No.\ DMS-2316687.%
}

\newcommand{\NN}{\mathbbm{N}}
\newcommand{\RR}{\mathbbm{R}}
\newcommand{\ZZ}{\mathbbm{Z}}

\newcommand{\LL}{\mathcal{L}}

\newcommand{\bT}{\mathbf{T}}
\newcommand{\bX}{\mathbf{X}}
\newcommand{\bph}{\boldsymbol{\ph}}
\newcommand{\bl}{\boldsymbol{\lambda}}
\newcommand{\bLam}{\boldsymbol{\Lambda}}
\newcommand{\bL}{\mathbf{L}}
\newcommand{\bs}{\boldsymbol{\sigma}}

\newcommand{\dF}{d_{\mathcal{F}}}

\newcommand{\ph}{\varphi}

\newcommand{\one}{\mathbbm{1}}
\newcommand{\eps}{\varepsilon}

\DeclareMathOperator{\diam}{diam}

\newcommand{\HC}{H} 
\newcommand{\Hexp}{\beta} 
\newcommand{\allA}{\ref{A:pre}--\ref{A:Hol}}
\newcommand{\TE}{\tau} 
\newcommand{\ip}[1]{\langle #1 \rangle}
\newcommand{\confact}{\rho}
\newcommand{\edc}{\gamma}

\begin{document}
\begin{abstract}
The Ruelle--Perron--Frobenius theorem is a powerful tool in the study of equilibrium measures and their statistical properties. 
We prove a nonstationary version of this theorem 
under general conditions involving 
an invariant sequence of real convex cones in function space. 
\end{abstract}

\maketitle

\ifdraft
\thispagestyle{firstpage}
\pagestyle{fancy}
\fi

\section{Introduction}
\label{sec:intro}

Consider a compact metric space $X$ and a continuous map $T\colon X\to X$ that is uniformly expanding, locally onto, and topologically exact.\footnote{
See \ref{A:exp} and \ref{A:exact} for precise definitions. The last two conditions follow from uniform expansion whenever $X$ is a compact connected manifold.} Given a H\"older continuous $\ph\colon X\to \RR$, which we call a  \emph{potential}, the Ruelle--Perron--Frobenius (RPF) operator associated to $\ph$ is the bounded linear operator $L\colon C(X) \to C(X)$ defined by
\begin{equation}\label{eqn:L}
(Lf)(x) = \sum_{y\in T^{-1}(x)} e^{\ph(y)} f(y)
\text{ for all } x\in X.
\end{equation}
The following classical theorem plays a central role in the study of thermodynamic formalism for the system $(X,T,\ph)$.
See \cite{Bowen} for a proof in the case when $(X,T)$ is a mixing subshift of finite type, or \cite[Chapter 4]{PU10} and \cite[\S12.1]{VO16} for the more general setting considered here.

\begin{RPF}
Writing $\lambda>0$ for the spectral radius of the operator $L$, the following are true.
\begin{enumerate}[leftmargin=*]
\item There is a unique Borel probability measure $m$ on $X$ such that $L^* m = \lambda m$.
\item There is a unique H\"older $h\colon X\to (0,\infty)$ such that $Lh = \lambda h$ and $\int h\,dm = 1$.
\item For any H\"older $f\colon X\to \RR$, we have $\lambda^{-n} L^n f \to (\int f \,dm) h$ exponentially fast.
\end{enumerate}
\end{RPF}

With $h,m$ as in the RPF theorem, the probability measure $\mu(E) = \int_E h\,dm$ is $T$-invariant and the unique equilibrium measure for $\ph$. That is, for every $T$-invariant $\nu$, we have $h_\nu(T) + \int\ph\,d\nu \leq \log\lambda$, with equality if and only if $\nu=\mu$.

\subsection{Nonstationary RPF theorems for uniformly expanding maps}\label{sec:unif-exp}

Now we replace the stationary system $(X,T,\ph)$ with a nonstationary setting.
Fix sequences of compact metric spaces $X_n$, continuous maps $T_n \colon X_n\to X_{n+1}$, and continuous potentials $\ph_n\colon X_n\to \RR$. 
We will investigate the behavior of the maps
\[
\bT_n^k = T_{n+k-1} \circ \cdots \circ T_{n+1} \circ T_n \colon X_n \to X_{n+k}.
\]
We assume for now that the sequence $\{(X_n,T_n,\ph_n)\}_n$ satisfies the properties listed below. In \S\ref{sec:abstract}, we discuss a more general setting where our results continue to hold.
Some results concern one-sided sequences ($n\in \NN\cup \{0\})$ and others concern two-sided sequences ($n\in \ZZ$); in the conditions below ``for all $n$'' will mean whichever of these two is relevant for any given result.

\begin{enumerate}[label={\upshape{(A\arabic{*})}},leftmargin=*]
\item\label{A:pre} \emph{Compact, continuous, and bounded degree:}
{Each $(X_n,d_n)$ is a compact metric space, and each $T_n \colon X_n \to X_{n+1}$ is continuous}. Moreover, there exists $D\in\NN$ such that for every $n$ and every $x\in X_{n+1}$, we have $1\leq \#T_n^{-1}(x) \leq D$.
\item\label{A:exp} \emph{Uniformly expanding and locally onto:}
there exist $\delta>0$ and $\confact\in (0,1)$ such that for all $n$ and all $x,x' \in X_n$
 satisfying $d_n(x,x') \leq \delta$, we have
\[
d_n(x,x') \leq \confact\, d_{n+1}(T_n(x),T_n(x')).
\]
Moreover, for any $x\in X_n$, we have $T_n(B(x,\delta)) \supset B(T_n x,\delta)$ in $X_{n+1}$.
\item\label{A:exact} \emph{Uniform topological exactness:}
with $\delta>0$ as in \ref{A:exp}, 
there exists  $\TE\in \NN$ such that for every $n$ and $x\in X_n$, we have $\bT_n^\TE(B_{d_n}(x,\delta)) = X_{n+\TE}$.
\item\label{A:Hol} \emph{Uniform H\"older continuity:}
there exist $\HC>0$ and $\Hexp\in (0,1)$ such that for all $n$ and all $x,x' \in X_n$ with $d_n(x,x') \leq \delta$, we have
\[
|\ph_n(x) - \ph_n(x')| \leq \HC d_n(x,x')^{\Hexp}.
\]
Moreover, there exists $V>0$ such that for all $n$, we have $\|\ph_n\|_\infty \leq V$.
\end{enumerate}

\vskip5pt
In this nonstationary setting, the RPF operator $L$ from \eqref{eqn:L} is replaced by a sequence of operators $L_n\colon C(X_n) \to C(X_{n+1})$ defined by
\begin{equation}\label{eqn:Ln}
(L_nf)(x) = \sum_{y\in T_n^{-1}(x)} e^{\ph_n(y)} f(y) \text{ for all } x \in X_{n+1}.
\end{equation}
We will write compositions of the RPF operators as follows:
\begin{align*}
\bL_n^k &= L_{n+k-1} \circ \cdots \circ L_{n+1} \circ L_n \colon C(X_n) \to C(X_{n+k}).
\end{align*}
Writing the nonstationary Birkhoff sums as $\bph_n^k(y) := \sum_{j=0}^{k-1} \ph_{n+j}(\bT_n^j y)$, we see that
\[
(\bL_n^kf)(x) = \sum_{y\in (T_n^k)^{-1}(x)} e^{\bph_n^k(y)} f(y) \text{ for all } x \in X_{n+k}.
\]
Writing $M(X_n)$ for the space of positive finite Borel measures on $X_n$,
we will also study the dual operators $L_n^* \colon M(X_{n+1}) \to M(X_n)$ defined by
\begin{equation}\label{eqn:Ln*}
\int_{X_n} f \,d(L_n^* m) = \int_{X_{n+1}} (L_n f) \,dm \quad\text{for all } f\in C(X_{n}),\ m\in M(X_{n+1}).
\end{equation}
This defines a positive Borel measure $L_n^* m$ by the Riesz representation theorem and positivity of the operator $L_n$: if $f>0$ then $L_n f>0$. Writing $\langle f,m \rangle = \int_{X_n} f\,dm$ when $f\in C(X_n)$ and $m\in M(X_n)$, \eqref{eqn:Ln*} can be rewritten as
\begin{equation}\label{eqn:Ln*2}
\langle f, L_n^* m \rangle = \langle L_n f , m \rangle.
\end{equation}

\vskip5pt
In Theorem \ref{thm:forward} below, we consider a one-sided infinite sequence ($n\in \NN \cup \{0\}$) and construct analogues of $\lambda$ and $m$, but not of $h$. In Theorem \ref{thm:two-sided}, we consider a two-sided infinite sequence ($n\in \ZZ$) and obtain a full analogue of the RPF theorem. Theorem \ref{thm:pseudo-inv} describes the analogue of the $T$-invariant measure $\mu$. Theorem \ref{thm:HolDep-intro} describes the H\"older continuous dependence of these various objects on the sequence $\{(T_n,\ph_n)\}_n$.

Our results will be proved using convex cones $\Lambda_n \subset C(X_n)$. Given $\{(X_n,T_n,\ph_n)\}_n$ satisfying \allA,
we write
\begin{equation}\label{eqn:Cn+}
C_n^+ := \{ f\in C(X_n) : f \geq 0,\ f\not\equiv 0\}
\end{equation}
for the space of nonnegative functions on $X_n$ that are not identically $0$. With $\Hexp,\delta$ as in \ref{A:Hol} and \ref{A:exp},
we consider for each $Q>0$ the following families of functions:
\begin{equation}\label{eqn:cone}
\Lambda_n(Q) := \{ f\in C_n^+: f(x) \leq e^{Q d_n(x,x')^\Hexp} f(x') \text{ whenever } d_n(x,x') \leq \delta \}.
\end{equation}
We will write $\one$ for the constant function that takes the value $1$, without explicitly indicating which space it is defined on, since this can be deduced from the context. See \S\ref{sec:prior} for a description of prior work on similar problems using this approach.

In \S\ref{sec:cones}, we describe the machinery of closed convex cones and the Birkhoff contraction theorem that we use to prove Theorems \ref{thm:forward} and \ref{thm:two-sided}. In particular, we formulate an abstract set of conditions on a sequence of operators on function spaces that can be verified in the setting of this section (Theorem \ref{thm:get-cones}) and that suffice to prove abstract analogues of the theorems here (Theorems \ref{thm:gen-forward}, \ref{thm:gen-two-sided}, and \ref{thm:continuity}).

\theoremstyle{plain}
\newtheorem{theo}{\textbf{Theorem}}[theorem]
\renewcommand*{\thetheo}{\textbf{\Roman{theo}}}

\begin{theo}\label{thm:forward}
Suppose $\{(X_n,T_n,\ph_n)\}_{n\geq 0}$ satisfy conditions \allA. 
Then for $Q$ sufficiently large and $\Lambda_n = \Lambda_n(Q)$ as in \eqref{eqn:cone}, the following are true.
\begin{enumerate}[label=\textbf{\upshape{[\Alph{*}]}},leftmargin=*]
\item\label{get-lm}
For every $n\geq0$ and every sequence of measures 
$\bs = \{\sigma_\ell \in M(X_\ell)\}_{\ell\geq n}$,
the following limits exist, and are independent of $\bs$: 
\begin{equation}\label{eqn:lnmn-0}
\lambda_n = \lim_{k\to\infty} \frac{\langle \bL_n^k \one, \sigma_{n+k} \rangle}{\langle \bL_{n+1}^{k-1} \one, \sigma_{n+k} \rangle}
\quad\text{and}\quad
m_n = \lim_{k\to\infty} \frac{(\bL_n^k)^* \sigma_{n+k}}{\langle \one, (\bL_n^k)^* \sigma_{n+k} \rangle}.
\end{equation}
The second limit is with respect to the weak* topology on $M(X_n)$.
\item\label{get-L*}
Each real number $\lambda_n$ is positive and each measure $m_n$ is a probability measure. They satisfy the relationship $L_n^* m_{n+1} = \lambda_n m_n$.
\item\label{get-m!}
The sequence $(m_n)_{n\geq 0}$ is the unique sequence of Borel probability measures $m_n \in M(X_n)$ with the property that $L_n^* m_{n+1}$ is a scalar multiple of $m_n$. 
\item\label{get-exp1}
The rate of convergence in \eqref{eqn:lnmn-0} is uniformly exponential in the following sense: there are constants $C_1>0$ and $\edc\in (0,1)$ such that for any $n\geq 0$, any $k\geq \TE$, any $\sigma\in M(X_{n+k})$, and any $f\in \Lambda_n$, we have
\begin{equation}\label{eqn:lnmn-exp}
\Big| \log \lambda_n - \log \frac{\langle \bL_n^k \one, \sigma \rangle}{\langle \bL_{n+1}^{k-1} \one, \sigma \rangle} \Big| \leq C_1 \edc^k 
\quad\text{and}\quad
\Big| \langle f, m_n \rangle - \frac{\langle f, (\bL_n^k)^*\sigma \rangle}{\langle \one, (\bL_n^k)^* \sigma \rangle} \Big|
\leq C_1 \|f\| \edc^k.
\end{equation}
\end{enumerate}
\end{theo}

For future reference, we point out that (as will be made clear in the proofs), the constants $Q,C_1,\edc$ in Theorem \ref{thm:forward} depend on the parameters $D,\delta,\rho,\tau,H,\beta,V$ in \allA, but not on the specific choice of the sequence $\{(T_n,\ph_n)\}_n$. The same is true in Theorem \ref{thm:two-sided} below.

It will be convenient to  define $\bl_n^k$ analogously to $\bT_n^k$ and $\bL_n^k$, writing
\[
\bl_n^k := \lambda_{n+k-1} \cdots \lambda_{n+1} \lambda_n = \prod_{j=0}^{k-1} \lambda_{n+j}.
\]
 Theorem \ref{thm:forward} contains no analogue of the eigenfunction $h$ from the stationary RPF theorem. 
For this, we must take a limit along backwards orbits, so we need to begin with a bi-infinite sequence.

\begin{theo}\label{thm:two-sided}
Consider a two-sided sequence  $\{(X_n,T_n,\ph_n)\}_{n\in \ZZ}$ satisfying conditions \allA. Then for all sufficiently large $Q$, once again taking $\Lambda_n = \Lambda_n(Q)$ as in \eqref{eqn:cone}, conclusions \ref{get-lm}--\ref{get-exp1} hold for all $n\in \ZZ$, and so do the following.
\begin{enumerate}[label=\textbf{\upshape{[\Alph{*}]}},leftmargin=*]
\setcounter{enumi}{4}
\item\label{get-h} 
For every $n\in \ZZ$, the following limit exists (in the uniform norm) for any sequence of functions $\{ f_\ell \in \Lambda_\ell \}_{\ell \leq n}$, and is independent of the choice of $f_\ell$: 
\begin{equation}\label{eqn:hn}
h_n = \lim_{k\to \infty} \frac{\bL_{n-k}^k f_{n-k}}{\bl_{n-k}^k 
\langle f_{n-k}, m_{n-k} \rangle}.
\end{equation}
\item\label{get-Lh} 
For every $n\in \ZZ$, 
the function $h_n$ is strictly positive, lies in $\Lambda_n$, and satisfies $\int h_n \,d m_n = 1$ and $L_n h_n = \lambda_n h_{n+1}$.
\item\label{get-h!}
The sequence $(h_n)_{n\in \ZZ}$ is the unique sequence of functions in $\Lambda_n$ such that $\int h_n \,dm_n = 1$ and $L_n h_n$ is a scalar multiple of $h_{n+1}$ for all $n\in \ZZ$. 
\item\label{get-exp2} 
The convergence in \eqref{eqn:hn} is uniformly exponential in the following sense: There exist $C_2>0$ and $\edc\in(0,1)$ such that for any $n\in\ZZ$, any $k\geq\TE$, and any 
$f\in \Lambda_{n-k}$, we have
\begin{equation}\label{eqn:hn-exp}
 \Big\|h_n-\dfrac{\bL_{n-k}^kf}{(\bl_{n-k}^{k})\langle f,m_{n-k}\rangle}\Big\|\leq C_2\edc^k,
 \end{equation}
\end{enumerate} 
\end{theo}

Observe that the conclusions \ref{get-lm}--\ref{get-exp2} do not involve $T_n$ or $\ph_n$ other than through $L_n$ and $\Lambda_n$. This suggests the approach that we will carry out in \S\ref{sec:abstract}, of formulating abstract conditions on the operators and cones that lead to the conclusions here, without reference to the specific dynamics and potentials that generate them.

We do not have a notion of measure-theoretic entropy or of equilibrium measure in the nonstationary setting, 
so there is no complete analogue of the statement from the stationary RPF theorem that multiplying the eigenfunction by the eigenmeasure produces the unique equilibrium measure. However,
we still have the following.

\begin{theo}\label{thm:pseudo-inv}
Consider a sequence of compact metric spaces $\{X_n\}_{n\in\ZZ}$,
 a sequence of continuous maps $\{T_n \colon X_n \to X_{n+1} \}_{n\in\ZZ}$, 
and a sequence of continuous potentials $\{\ph_n \colon X_n \to \RR\}$. Let $L_n \colon C(X_n) \to C(X_{n+1})$ be the operators defined in \eqref{eqn:Ln}. Suppose that $\lambda_n \in (0,\infty)$, $m_n \in M(X_n)$, and $h_n \in C_n^+$ are such that
\[
L_n h_n = \lambda_n h_{n+1},\quad
L_n^* m_{n+1} = \lambda_n m_n,\quad
\text{and }
\langle h_n, m_n \rangle = 1
\quad\text{for every } n\in \ZZ.
\]
Then the measures $\mu_n \in M(X_n)$ defined by $\mu_n(E) = \int_E h_n\,dm_n$ have the following pseudo-invariance property:
$(T_n)_*\mu_n=\mu_{n+1}$ for all $n\in\ZZ$.

If in addition we have $h_n>0$, then considering the normalized potential
\begin{equation}\label{eqn:normalized}
\Tilde{\varphi}_n :=\varphi_n+\log h_n-\log h_{n+1}\circ T_n-\log\lambda_n
\end{equation}
and writing $\Tilde{L}_n$ for the associated operator as in \eqref{eqn:Ln}, so that
\begin{equation}\label{eqn:iterated-cocycle}
(\widetilde{\bL}_n^k f)
= \frac {\bL_n^k(h_n f)}{\bl_n^k h_{n+k}} .
\end{equation}
for every $n\in\ZZ$ and $k\in\NN$,
then we have
\begin{equation}\label{eqn:tilde-L}
\widetilde{L}_n \one = \one
\quad\text{and}\quad
\widetilde{L}_n^* \mu_{n+1} = \mu_n
\quad\text{for every $n\in \ZZ$.}
\end{equation}
\end{theo}


Finally, we describe how $\lambda_n$, $m_n$, and $h_n$ depend on the maps $T_n$ and the potentials $\ph_n$.
Given a sequence $(X_n,d_n)$ of compact metric spaces, fix constants $D\in\NN$, $\delta>0$ and $\rho\in(0,1)$, $\tau\in\NN$, and $V,H,\beta>0$ corresponding to those in Assumptions \allA, respectively. Let $\mathcal{F}=\mathcal{F}(D,\delta,\rho,\tau,H,\beta,V)$ denote the set of pairs $(\bT,\bph) = ((T_n,\ph_n))_{n\in \ZZ}$ that satisfy \allA\ for these constants. Similarly, let $\mathcal{F}^+$ denote the set of pairs $(\bT,\bph) = ((T_n,\ph_n))_{n\geq 0}$ satisfying \allA.

Consider the following $[0,\infty]$-valued metric $\dF$ on $\mathcal{F}$:
\begin{equation}
\label{eqn:dF}
\dF((\bT,\bph),(\bT',\bph')) := \sum_{n\in \ZZ} 2^{-|n|} 
\Big(\sup_{x\in X_n} d_{n+1}(T_n(x), T_n'(x))
+ \|\ph_n - \ph_n'\|_\infty \Big).
\end{equation}
Define $d_{\mathcal{F}^+}$ on $\mathcal{F}^+$ similarly, using a sum over $n\geq 0$ instead of over all $n\in \ZZ$.

\begin{theo}\label{thm:HolDep-intro}
The map $(\bT,\bph) \mapsto (\lambda_0,m_0,h_0)$ is locally H\"older continuous with respect to $\dF$, in the following sense:
\begin{enumerate}[label=\upshape{(\alph{*})},leftmargin=*]
\item the real number $\lambda_0$ depends locally  H\"older continuously on $(\bT,\bph)$;
\item for every $f \in \Lambda_0$, the integral $\int_{X_0} f \,dm_0$ depends locally  H\"older continuously on $(\bT,\bph)$;
\item the function $h_0 \in \Lambda_0 \subset C(X_0)$ depends locally H\"older continuously on $(\bT,\bph)$ with respect to the uniform metric on $C(X_0)$.
\end{enumerate}
Shifting the sequence gives local H\"older continuity of $(\lambda_n,m_n,h_n)$ for every $n$.
\end{theo}

\subsection{Prior work}\label{sec:prior}

The set $\Lambda_n(Q) \subset C(X_n)$ in \eqref{eqn:cone} is an example of a \emph{closed convex cone}. In \S\ref{sec:cones}, we will provide precise definitions and describe how the \emph{Hilbert projective metric} can be used to formulate conditions under which a RPF theorem holds. In this section, we briefly describe some of the history of these ideas, and how they fit into thermodynamic formalism more broadly. It is worth pointing out that the RPF theorem was originally proved in the 1970s by other methods, not via convex cones \cite{Bowen}.

Garrett Birkhoff's 1957 paper \cite{gB57} used the Hilbert projective metric to prove the Perron--Frobenius theorem for positive matrices.
The idea of extending this approach to the (infinite-dimensional) Ruelle--Perron--Frobenius Theorem can be traced back to work of Ferrero and Schmitt \cite{FS79,FS81}. 
Cone methods became more widely used in dynamical systems and ergodic theory following Liverani's 1995 work on decay of correlations \cite{cL95}; see Baladi's book \cite{vB00} for an overview of the subject around this time.
We also refer to \cite{EN95a,EN95} for another proof of Birkhoff's contraction theorem and an application to an RPF-type result in a more abstract setting, and to \cite{fN04} for some explicit computations of the decay rate of correlations using cone techniques.

Moving beyond the uniformly hyperbolic setting,
cone techniques were applied to nonuniformly expanding systems by Viana and Varandas \cite{VV10} and by Castro and Varandas \cite{CV13}. 
In \cite{gH23}, the second author of the present paper developed nonstationary versions of some of these results to study nonuniformly expanding skew products, obtaining a description of the conditional measures of equilibrium states along fibers.
(For uniformly expanding skew products, results along these lines had been obtained by Denker and Gordin \cite{DG99}.)
Although our Theorems \ref{thm:forward} and \ref{thm:two-sided} are formulated for uniformly expanding maps, 
we set up the proofs in such a way that our general results (Theorems \ref{thm:gen-forward} and \ref{thm:gen-two-sided}) can be applied to these nonuniformly expanding examples as well (see \S\ref{sec:abstract}). Indeed, one motivation for the present work was to clearly identify the conditions that must be verified in order to apply Birkhoff's contraction theorem to Hilbert's projective metric, with the goal of paving the way for future generalizations.

The skew products studied in \cite{DG99,gH23} are closely connected to random dynamical systems, where one works with respect to a given invariant measure for the base transformation.
Thermodynamic formalism in random dynamics has been studied at least since the early 1990s; see for example \cite{tB92,K92,KK96}.
This direction was developed further in the book by Mayer, Skorulski, and Urba\'nski \cite{MSU}. Some of the non-uniformly expanding results mentioned above have been extended to the random setting by Stadlbauer, Suzuki, and Varandas \cite{SSV21}. 
Our setting is \emph{a priori} more general than the corresponding results for skew products and random transformations, since we will impose no requirement that our sequence of maps $T_n$ comes from any base dynamics (see \S\ref{sec:abstract}).

We conclude by pointing out that here we restrict our attention to real potential functions and real cones, but cone techniques have also been developed for complex transfer operators and complex projective metrics \cite{hR10,yH20,yH22,yH23}.

\subsection{Outline of the paper}

In \S\ref{sec:cones}, we describe the general machinery of convex cones and the Hilbert metric, and state abstract versions of Theorems \ref{thm:forward}, \ref{thm:two-sided}, and 
\ref{thm:HolDep-intro}.
In \S\ref{sec:get-cones}, we describe how Conditions \allA\ imply the abstract cone conditions laid out in \S\ref{sec:cones}. In \S\ref{sec:gen-cones}, we recall some general results on the cone of non-negative functions, which are used in \S\S\ref{sec:thm1}--\ref{sec:thm2} when we prove the abstract versions of Theorems \ref{thm:forward}--\ref{thm:two-sided}.
In \S\ref{sec:convergence-pseudo}, we prove Theorem \ref{thm:pseudo-inv} on convergence to pseudo-invariant measures, and in \S\ref{sec:Holder}, we prove Theorem \ref{thm:HolDep-intro} (and its generalization) on H\"older dependence.

\section{A general result using convex cones}\label{sec:cones}

To describe the strategy for the proofs of Theorems \ref{thm:forward} and \ref{thm:two-sided}, we first recall, in \S\ref{sec:hilbert}, the machinery of convex cones and the Birkhoff contraction theorem.
Then in \S\ref{sec:abstract}, we formulate an abstract set of cone conditions that can be deduced from \allA, and that lead to the conclusions of Theorems \ref{thm:forward} and \ref{thm:two-sided}.

\subsection{Convex cones and the Hilbert metric}\label{sec:hilbert}

See \cite{cL95,fN04,VO16} for more details and proofs of the results in this section that are not proved here.

\begin{definition}
Let $\Omega$ be a Banach space. A subset $\Lambda\subset \Omega\setminus\{0\}$ is a \emph{cone} if  $t \Lambda = \Lambda$ for all $t>0$. It is \emph{convex} if for every $f,g\in \Lambda$, we have $f + g\in \Lambda$. A \emph{closed convex cone} is a convex cone $\Lambda$ for which $\Lambda \cup \{0\}$ is norm-closed.
\end{definition}

\begin{example}
The sets $C_n^+$ and $\Lambda_n(Q)$ from 
equations \eqref{eqn:Cn+} and \eqref{eqn:cone}, respectively,
are closed convex cones.
\end{example}

We can define a partial ordering $\preceq$ on a closed convex cone $\Lambda$ by saying that $f\preceq g$ if and only if $g-f\in \Lambda\cup\{0\}$. 
Let 
\begin{equation}\label{eqn:AB}
A(f,g)=\sup\{t>0\colon tf\preceq g\}\ \text{and\ } B(f,g)=\inf\{t>0\colon g\preceq tf\}.
\end{equation}
If $g$ is a scalar multiple of $f$, then $A=B$. Otherwise, the open interval $(A,B)$ is the interior of the interval of values $t \in (0,\infty)$ for which $tf$ and $g$ are not comparable with respect to $\preceq$, and we obtain a semi-metric by taking the multiplicative size of this interval:

\begin{definition}\label{def:Hilbert}
The \emph{Hilbert projective metric} $\Theta$ with respect to $\Lambda$ is defined as 
\begin{equation}\label{eqn:Theta}
\Theta(f,g)=\log\frac{B(f,g)}{A(f,g)}.
\end{equation}
\end{definition}

Note that $\Theta(\alpha f,\beta g) = \Theta(f,g)$ for all $\alpha,\beta>0$, and in particular $\Theta(f,g) = 0$ if and only if there exists $t>0$ such that $f = t g$. The following contraction theorem due to Birkhoff will be essential to our arguments in this paper. See \cite[Theorem 1.1]{cL95} or \cite[Proposition 12.3.6]{VO16} for recent presentations of the proof.

\begin{theorem}[Birkhoff Contraction Theorem \cite{gB57}]\label{thm:Birkhoff}
Let $\Omega_1,\Omega_2$ be Banach spaces, and $\Lambda_i\subset \Omega_i$ closed convex cones for $i=1,2$. Let $\Theta_i$ be the corresponding Hilbert projective metrics.
If a linear operator $L\colon \Omega_1\to \Omega_2$ has the property that 
$L(\Lambda_1)\subset \Lambda_2$, then for all $f,g\in \Lambda_1$, we have
\[
\Theta_{2}(Lf,Lg)\leq \tanh\Big(\frac{\Delta}{4}\Big)\Theta_{1}(f,g),
\] 
where $\Delta:=\diam_{2}(L \Lambda_1)=\sup\{\Theta_{2}(Lf,Lg) : f,g \in \Lambda_1\}$ and $\tanh\infty=1$.

In particular, if $\Delta<\infty$, then $L|_{\Lambda_1}$ is a contraction by $\tanh(\Delta/4) < 1$.
\end{theorem}

\subsection{Abstract cone conditions for an RPF theorem}\label{sec:abstract}

The following result applies to both one-sided and two-sided sequences, and will be proved in \S\ref{sec:get-cones}.

\begin{theorem}\label{thm:get-cones}
Let $\{(X_n,T_n,\ph_n)\}_n$ be a sequence satisfying \allA, and fix a constant $Q > \HC \confact^\Hexp(1-\confact^\Hexp)^{-1}$, where $\confact\in (0,1)$ is the contraction factor from \ref{A:exp} and $\HC,\Hexp$ are the H\"older parameters from \ref{A:Hol}.

Then the cones $\Lambda_n := \Lambda_n(Q) \subset C(X_n)$ defined in \eqref{eqn:cone}, and the RPF operators $L_n \colon C(X_n) \to C(X_{n+1})$ defined in \eqref{eqn:Ln}, satisfy the following conditions, taking $\TE \in \NN$ as in condition \ref{A:exact}
and writing $\Theta_n$ for the Hilbert projective metric on $\Lambda_n$:
\begin{enumerate}[label=\upshape{(C\arabic{*})},leftmargin=*]
\setcounter{enumi}{-1}
\item\label{C0}
Each $\Lambda_n-\Lambda_n$ is dense in $C(X_n)$ (in the uniform metric).
\item\label{C1} 
Each $\Lambda_n \subset C(X_n)$ is a closed convex cone 
such that $\one \in \Lambda_n \subset C_n^+$.
\item\label{C2} 
Each $L_n \colon C(X_n) \to C(X_{n+1})$ is a bounded linear operator with
$L_n\Lambda_n\subset\Lambda_{n+1}$, and there exists $\kappa\geq 1$
such that for all $n$, we have
\begin{gather*}
\kappa^{-1} \leq L_n\one(x)\leq \kappa
\text{ for all } x\in X_{n+1},
\text{ and} \\
\|L_n f\|_\infty \leq \kappa \|f\|_\infty
\text{ for all } f\in \Lambda_n.
\end{gather*}
Moreover, there exists $\Delta \in (0,\infty)$ and $\TE\in\NN$ such that for every $n$, we have
\[
\diam_{\Theta_{n+\TE}} \bL_n^\TE\Lambda_n \leq \Delta
\quad\text{and}\quad
\sup \bL_n^\TE\one \leq e^\Delta \inf \bL_n^\TE\one.
\]
\end{enumerate}
\end{theorem}

It turns out that Conditions \ref{C0}--\ref{C2} carry all the information we need about the dynamics in order to apply Birkhoff's contraction theorem and deduce the RPF theorem. In fact, we can weaken \ref{C0} slightly: the following result is proved in \S\ref{sec:thm1}.

\begin{theorem}\label{thm:gen-forward}
Consider a sequence of compact metric spaces $\{ X_n \}_{n\geq 0}$,
a sequence of bounded linear operators $\{L_n \colon C(X_n) \to C(X_{n+1}) \}_{n\geq 0}$,
and a family of closed convex cones
$\{ \Lambda_n \subset C_n^+ \subset C(X_n) \}_{n\geq 0}$ 
satisfying \ref{C1}, \ref{C2}, and the following:
\begin{enumerate}[label=\upshape{(C\arabic{*})},leftmargin=*]\setcounter{enumi}{2}
    \item\label{C3} 
Writing  $\Lambda_n^p := \{ f \in C_n^+ : \bL_n^p f \in \Lambda_{n+p} \}$ and $\Lambda_n^\infty := \bigcup_{p=0}^\infty \Lambda_n^p$, the set
$\Lambda_n^\infty - \Lambda_n^\infty$ is (uniformly) dense in $C(X_n)$ for every $n$.
\end{enumerate}
Then conclusions \ref{get-lm}--\ref{get-exp1} of Theorem \ref{thm:forward} hold for every $n\geq 0$.
\end{theorem}

Theorems \ref{thm:get-cones} and \ref{thm:gen-forward} together imply Theorem \ref{thm:forward}.
Similarly, Theorem \ref{thm:two-sided} is a consequence of Theorem \ref{thm:get-cones} and the following result, which is proved in \S\ref{sec:thm2}.

\begin{theorem}\label{thm:gen-two-sided}
Consider a sequence of compact metric spaces $\{ X_n \}_{n\in \ZZ}$,
a sequence of bounded linear operators $\{L_n \colon C(X_n) \to C(X_{n+1}) \}_{n\in \ZZ}$,
and a family of closed convex cones $\{ \Lambda_n \subset C_n^+ \subset C(X_n) \}_{n\in \ZZ}$ 
such that \ref{C1}, \ref{C2}, and \ref{C3} are satisfied.
Then conclusions \ref{get-lm}--\ref{get-exp2} of Theorems \ref{thm:forward} and \ref{thm:two-sided} hold for every $n\in \ZZ$.
\end{theorem}

We emphasize that Theorems \ref{thm:gen-forward} and \ref{thm:gen-two-sided} do not assume that the operators $L_n$ come from any specific dynamics $T_n$ or potentials $\ph_n$. In particular, one can apply these theorems beyond the uniformly expanding setting: for example, in the nonuniformly expanding skew products studied previously by the second author, any sequence of fiber transformations over a trajectory of the base dynamics satisfies \ref{C1}--\ref{C3} \cite[Lemma 3.4]{gH23}, and thus Theorem \ref{thm:gen-forward} applies.

We conclude with a H\"older continuity result along the lines of Theorem \ref{thm:HolDep-intro}.
Fix a sequence $\bX = \{(X_n,d_n)\}_{n\in\ZZ}$ of compact metric spaces,
and a sequence of closed convex cones $\Lambda_n \subset C_n^+ \subset C(X_n)$.
Consider for each $n$ the following norm on the set of linear operators $L \colon C(X_n) \to C(X_{n+1})$:
\begin{equation}\label{eqn:norm-Lambda}
\|L\|_{\Lambda_n} := \sup \big\{ \|Lf\|_\infty : f\in \Lambda_n,\ \|f\|_\infty = 1 \big\}.
\end{equation}
Writing $\bLam := \{\Lambda_n\}_n$ for the sequence of cones, let $\LL = \LL(\bX,\bLam)$ denote the space of sequences $\bL = \{L_n \colon C(X_n) \to C(X_{n+1}) \}_{n\in \ZZ}$ of linear operators that are bounded in these norms.
Then the following defines a $[0,\infty]$-valued metric on $\LL$:
\begin{equation}\label{eqn:dL}
d_\LL(\bL,\bL') := \sum_{n\in \ZZ} 2^{-|n|} \|L_n - L_n'\|_{\Lambda_n}.
\end{equation}
Now fix parameters $\kappa\geq 1$, $\Delta>0$, and $\tau\in \NN$, and let $\LL_0(\kappa,\Delta,\tau) \subset \LL$ denote the set of $\bL$ for which
\ref{C1}--\ref{C3} are satisfied
with this family of cones and with these values of $\kappa,\Delta,\tau$. Define $\LL^+$, $d_{\LL^+}$, and $\LL_0^+$ similarly using forward sequences $\{L_n\}_{n\geq 0}$.

\begin{theorem}\label{thm:continuity}
The sets $\LL_0$ and $\LL_0^+$
have the following continuity results with respect to $d_\LL$ and $d_{\LL^+}$.
\begin{enumerate}[label=\textup{(\alph{*})},leftmargin=*]
\item The map $\LL_0^+ \to (0,\infty)$ taking $\bL \mapsto \lambda_0$ is locally H\"older continuous.
\item For all $f \in \Lambda_0$,
the map $\LL_0^+ \to \RR$ taking $\bL \mapsto \int f \,dm_0$ is locally H\"older continuous.
\item The map $\LL_0 \to C(X_0)$ taking $\bL \mapsto h_0$ is locally H\"older continuous with respect to the uniform norm on $C(X_0)$.
\end{enumerate}
By shifting the sequences, it follows that $(\lambda_n,m_n,h_n)$ depend locally H\"older continuously on $\bL$ for every $n$ (with H\"older constants depending on $n$).
\end{theorem}

\section{Deducing cone conditions}\label{sec:get-cones}

This section contains the proof of Theorem \ref{thm:get-cones}. 
Condition \ref{C1} is immediate from \eqref{eqn:cone}.
For Condition \ref{C0}, we prove the following.

\begin{lemma}\label{lem:get-C3}
For every $Q>0$ and $n\in \ZZ$, the set $\Lambda_n(Q) - \Lambda_n(Q)$ contains all $\Hexp$-H\"older functions on $X_n$. In particular, this set is dense in $C(X_n)$.
\end{lemma}

We prove Lemma \ref{lem:get-C3} via a general fact about H\"older continuous functions. Recall that given a compact metric space $X$, a function $f\colon X\to \RR$ is $\Hexp$-H\"older continuous function if its H\"older seminorm is finite; i.e. 
\[
|f|_\Hexp := \sup \bigg \{ \frac{|f(x) - f(y)|}{d(x,y)^\Hexp} : x,y\in X,\ x\neq y \bigg\}<\infty.
\]

\begin{lemma}\label{lem:log-norm}
Given a compact metric space $X$ and a $\Hexp$-H\"older continuous function $f\colon X\to \RR$, for every $c>-\inf f$, we have
\begin{equation}\label{eqn:log-norm}
|\log(f+c)|_\Hexp \leq |f|_\Hexp (c+\inf f)^{-1}.
\end{equation}
\end{lemma}
\begin{proof}
First observe that given any $s > r \geq a > 0$, we have
\begin{equation}\label{eqn:logrs}
\log s - \log r = \int_r^s \frac 1t \,dt \leq \frac{s-r}{a}.
\end{equation}
Writing $f_c := f+c$, we see that $f_c(x) \geq c + \inf f > 0$ for every $x\in X$. Given any $x,y\in X$ with $x\neq y$, we can use \eqref{eqn:logrs} with $a = c + \inf f$ to conclude that
\[
|\log f_c(x) - \log f_c(y)|
\leq \frac {|f_c(x) - f_c(y)|}{c+\inf f}
= \frac{|f(x) - f(y)|}{c+\inf f} \leq \frac{|f|_\Hexp d(x,y)^\Hexp}{c+\inf f}.\qedhere
\]
\end{proof}

\begin{proof}[Proof of Lemma \ref{lem:get-C3}]
Observe that if $g \colon X_n \to (0,\infty)$ is a $\Hexp$-H\"older continuous function with $|\log g|_\Hexp \leq Q$, then for all $x,y\in X$ with $d(x,y) \leq \delta$, we have $\log g(y) \leq \log g(x) + Qd(x,y)^\Hexp$, so $g \in \Lambda_n(Q)$.

Now given any $\beta$-H\"older continuous function $f\colon X_n\to \RR$, let $c>\max(0,-\inf f)$ be sufficiently large that $|f|_\Hexp (c+\inf f)^{-1} \leq Q$. By Lemma \ref{lem:log-norm}, this implies that $|\log(f+c)|_\Hexp \leq Q$, and the previous paragraph implies that $f+c \in \Lambda_n(Q)$. Since $\Lambda_n(Q)$ contains all positive constant functions, we have $c \in \Lambda_n(Q)$, and thus $f = (f+c) - c \in \Lambda_n(Q) - \Lambda_n(Q)$.

This proves that the set $\Lambda_n(Q) - \Lambda_n(Q)$ contains all $\Hexp$-H\"older continuous functions on $X_n$, and since every function in $C(X_n)$ can be uniformly approximated by such functions, we are done.
\end{proof}

Now we turn our attention to \ref{C2}. For this, we use a series of lemmas that follow the usual approach to proving finite Hilbert diameter for the image of a log-H\"older cone under the RPF operator associated to an expanding map and a H\"older potential; see \cite{cL95,fN04,VO16} for these arguments in the stationary setting. 
The first lemma relies on the expanding and locally  onto properties in \ref{A:exp}.

\begin{lemma}\label{lem:Q-kappa}
Given $Q>0$, let $S = S(Q) := \confact^\Hexp(\HC+Q)$. Then 
$L_n \Lambda_n(Q) \subset \Lambda_{n+1} (S)$ for every $n$. Moreover, for every $x\in X_{n+1}$, we have
\begin{equation}\label{eqn:Ln1}
e^{-V} \leq L_n\one(x) \leq D e^V.
\end{equation}
\end{lemma}

\begin{proof}
Let $\delta>0$ be as in \ref{A:exp}. Given $f\in \Lambda_n(Q)$, fix $x,x' \in X_{n+1}$ such that $d_{n+1}(x,x') \leq \delta$. We will prove that $L_nf(x) \leq e^{Sd_{n+1}(x,x')^\Hexp} L_nf(x')$.

Given any $y\in T_n^{-1}(x)$, by \ref{A:exp}, there exists $y' \in T_n^{-1}(x') \cap B(y,\delta)$ such that $d_n(y,y') \leq \confact d_{n+1}(x,x')$. Together with the H\"older estimates in \ref{A:Hol} and the definition of $\Lambda_n(Q)$ in \eqref{eqn:cone}, this implies that
\begin{align*}
\ph_n(y) &\leq \ph_n(y') + \HC d_n(y,y')^\Hexp
\leq \ph_n(y') + \HC \confact^{\Hexp} d_{n+1}(x,x')^\Hexp, \\
\log f(y) &\leq \log f(y') + Q d_n(y,y')^\Hexp
\leq \log f(y') + Q \confact^{\Hexp} d_{n+1}(x,x')^\Hexp.
\end{align*}
Using these inequalities in the definition of $L_n$, we get
\begin{align*}
L_nf(x) = \sum_{y\in T_n^{-1}x} e^{\ph_n(y)} f(y)
&\leq \sum_{y' \in T_n^{-1}x'} e^{\ph_n(y')+ \HC \confact^{\Hexp} d_{n+1}(x,x')^\Hexp} f(y') e^{Q \confact^{\Hexp} d_{n+1}(x,x')^\Hexp} \\
&= L_nf(x') e^{\confact^{\Hexp}(\HC+Q) d_{n+1}(x,x')^\Hexp}.
\end{align*}
We conclude that $L_n f \in \Lambda_{n+1}(\confact^{\Hexp}(\HC+Q))$.

To prove \eqref{eqn:Ln1}, it suffices to use the definition $L_n\one(x) = \sum_{y\in T_n^{-1}(x)} e^{\ph_n(y)}$ together with the observation that \ref{A:pre} gives $1\leq \#T_n^{-1}(x) \leq D$, and by \ref{A:Hol}, each $y\in T_n^{-1}(x)$ must satisfy $-V \leq \ph_n(y) \leq V$.
\end{proof}

With $S=S(Q)$ as in Lemma \ref{lem:Q-kappa}, a simple computation shows that 
\begin{equation}\label{eqn:S<Q}
\text{$S < Q$ if and only if $Q > \HC \confact^\Hexp(1-\confact^\Hexp)^{-1}$}.
\end{equation}
This motivates our choice of $Q$ in Theorem \ref{thm:get-cones}.

Given $R>0$, define for each $n$ the cone
\begin{equation}\label{eqn:CR}
C_n^+(R) := \{ f\in C_n^+ : \sup f \leq R \inf f \}.
\end{equation}
We will see in \S\ref{sec:gen-cones} that this is the ball of projective radius $R$ around $\one$ in $C_n^+$. The next lemma relies on the topological exactness property in \ref{A:exact}, as well as the uniform bounds on number of preimages in \ref{A:pre}, and on the range of $\ph_n$ in \ref{A:Hol}.

\begin{lemma}\label{lem:+R}
Given $Q>0$, let $R = R(Q) := D^\TE e^{\TE V} e^{Q\delta^\Hexp}$, where $D,\TE,V,\delta,\Hexp$ are as in \allA. Then for every $n$, we have
$\bL_n^\TE \Lambda_n(Q) \subset C_{n+\TE}^+(R)$.
\end{lemma}
\begin{proof}
By \ref{A:Hol}, there exists $V\in (0,\infty)$ such that $\sup \ph_n - \inf \ph_n \leq V$ for all $n$, and thus $\sup \bph_n^k - \inf \bph_n^k \leq kV$ for all $n,k$.
Fix $n\in \ZZ$ and let 
$a_n = \inf \bph_n^\TE$,
so $\bph_n^\TE(y) \in [a_n, a_n + \TE V]$ for all $y\in X_n$.

Given $f\in \Lambda_n(Q)$, let $y_0 \in X_n$ be such that $f(y_0) = \sup f$. Note that this value is positive since $f\not\equiv 0$. By the definition of $\Lambda_n(Q)$, we have
\[
f(y) \geq e^{-Q \delta^\Hexp} f(y_0)
\quad\text{for all } y\in B_{d_n}(y_0,\delta).
\]
Given $x\in X_{n+\TE}$, condition \ref{A:exact} implies that there exists $y\in B_{d_n}(y_0,\delta) \cap (\bT_n^\TE)^{-1}(x)$,
and thus
\begin{equation}\label{eqn:Lfx}
\bL_n^\TE f(x) \geq e^{\bph_n^\TE(y)} f(y) \geq e^{a_n} e^{-Q \delta^\Hexp} \sup f.
\end{equation}
As in \ref{A:pre}, let $D\in \NN$ be an upper bound on the number of preimages of any point under any $T_n$. Then given any $x' \in X_{n+\TE}$, we have $\#(\bT_n^\TE)^{-1}(x') \leq D^\TE$, so
\begin{equation}\label{eqn:Lfx'}
\bL_n^\TE f(x') 
= \sum_{y \in (\bT_n^\TE)^{-1}(x')} e^{\bph_n^\TE(y)} f(y)
\leq D^\TE e^{a_n + \TE V} \sup f.
\end{equation}
Combining \eqref{eqn:Lfx'} and \eqref{eqn:Lfx} gives
\[
\frac{\sup \bL_n^\TE f}{\inf \bL_n^\TE f}
\leq D^\TE e^{\TE V} e^{Q \delta^\Hexp}.\qedhere.
\]
\end{proof}

Combining the previous two lemmas and \eqref{eqn:S<Q}, we see that 
under the hypotheses of Theorem \ref{thm:get-cones}, we have 
$\bL_n^\TE \Lambda_n(Q) \subset \Lambda_{n+\tau}(S) \cap C_{n+\tau}^+(R)$, where $S=S(Q) < Q$ and $R = R(Q)$ take the values stated in the lemmas. Then \ref{C2} follows from the next lemma.

\begin{lemma}\label{lem:Delta}
Given any $Q > S > 0$, any $R>0$, let $\Delta := 2\log(\frac{Q+S}{Q-S} \cdot R)$.
For every $f,g\in \Lambda_n(S) \cap C_n^+(R)$, we have $\Theta_n(f,g) \leq \Delta$.
\end{lemma}
\begin{proof}
Recalling Definition \ref{def:Hilbert}, we need to determine for which values of $t\in (0,\infty)$ we have $g - tf \in \Lambda_n(Q)$ or $tf-g \in \Lambda_n(Q)$. Membership in $\Lambda_n(Q)$ is determined in terms of the following set of pairs of nearby points:
\[
\mathcal{P} := \{ (x,y) : x,y\in X_n,\ y\neq x,\ d_n(x,y) \leq \delta \}.
\]
Given $(x,y) \in \mathcal{P}$, define a linear functional $\ell_{x,y} \colon C(X_n) \to \RR$ by
\[
\ell_{x,y}(f) = e^{Qd_n(x,y)^\Hexp} f(x) - f(y),
\]
and observe that the definition of $\Lambda_n(Q)$ in \eqref{eqn:cone} is equivalent to
\[
\Lambda_n(Q) := \{ f \in C(X_n)\setminus\{0\} : f\geq 0 \text{ and }
\ell_{x,y}(f) \geq 0 \text{ for all } (x,y) \in \mathcal{P} \}.
\]
Since the functions $\ell_{x,y}$ are linear, we have
$\ell_{x,y}(g-tf) = \ell_{x,y}(g) - t\ell_{x,y}(f)$ for all $t$, so $g-tf\in \Lambda_n(Q)$ if and only if
\[
\ell_{x,y}(g) \geq t\ell_{x,y}(f) \text{ for all } (x,y)\in \mathcal{P},
\text{ and } g(x) \geq tf(x) \text{ for all } x\in X_n.
\]
{It follows that $A(f,g)$ is the infimum of the set
\[
\bigg\{ \frac{\ell_{x,y}(g)}{\ell_{x,y}(f)} : (x,y) \in \mathcal{P} \bigg\}
\cup
\bigg\{ \frac{g(x)}{f(x)} : x\in X_n \bigg\},
\]
and a similar argument shows that $B(f,g)$ is the supremum of this set.} Thus Definition \ref{def:Hilbert} gives
\[
\Theta_n(f,g) = \log \sup \bigg\{
\frac{\ell_{x,y}(g)}{\ell_{x,y}(f)}
\frac{\ell_{x',y'}(f)}{\ell_{x',y'}(g)} : (x,y), (x',y') \in \mathcal{P} \bigg\}
\cup \bigg\{ \frac{g(x) f(y)}{f(x) g(y)} : x,y \in X_n \bigg\}.
\]
Given $(x,y) \in \mathcal{P}$, let $r := d_n(x,y)^\Hexp \in [0,\delta^\Hexp]$, and observe that since $g\in \Lambda_n(S)$, we have $g(y) \geq e^{-Sr} g(x)$, leading to
\[
\ell_{x,y}(g) = e^{Qr}g(x) - g(y) \leq g(x) \big( e^{Qr} - e^{-Sr} \big).
\]
Similarly, $f(y) \leq e^{Sr} f(x)$, so
$\ell_{x,y}(f) \geq f(x)(e^{Qr} - e^{Sr})$, giving
\[
\frac{\ell_{x,y}(g)}{\ell_{x,y}(f)}
\leq \frac{g(x) (e^{Qr} - e^{-Sr})}{f(x) (e^{Qr} - e^{Sr})}
= M(r) \frac{g(x)}{f(x)},
\quad\text{where }
M(r) := \frac{e^{Qr} - e^{-Sr}}{e^{Qr} - e^{Sr}}.
\]
Some routine differentiation
shows that $M(r)$ is decreasing in $r$, and that $M(0) := \lim_{r\to 0^+} M(r) = \frac{Q+S}{Q-S}$. Using this fact, together with a similar inequality when the roles of $f,g$ are reversed, shows that
\[
\Theta_n(f,g) \leq \log \Big( M(0)^2 \sup_{x,y\in X_n} \frac{g(x) f(y)}{f(x) g(y)} \Big)
\leq 2 \log \Big( \frac{Q+S}{Q-S} \cdot R \Big).\qedhere
\]
\end{proof}


\section{The cone of nonnegative functions}\label{sec:gen-cones}

In this section, we collect some general facts about convex cones of functions that will be used in the proofs in \S\ref{sec:thm1} and \S\ref{sec:thm2}. Most of these concern the cone of nonnegative functions and the relationship between its Hilbert metric and the uniform norm.

Let $X$ be a compact metric space, and $C^+ = \{ f\in C(X) : f\geq0,\ f\not\equiv 0\}$ its cone of nonnegative functions. Write $\Theta_+$ for the Hilbert projective metric \eqref{eqn:Theta} associated to $C^+$.

\begin{lemma}[{\cite[Lemma 1.4]{EN95}, \cite[Proposition 4.14]{fN04}, \cite[Example 12.3.5]{VO16}}]
\label{lem:hilbert}
The Hilbert distance between any $f,g\in C^+$ is 
    \begin{equation}\label{eqn:Theta+}
    \Theta_+(f,g) =\log \sup_{\mu,\nu\in M(X)} \frac{\langle g,\mu\rangle \langle f,\nu\rangle}{\langle g,\nu\rangle \langle f,\mu\rangle}
= \log \sup_{x,y \in X} \frac{ g(x) f(y)}{g(y) f(x)}.
    \end{equation}
\end{lemma}

An important special case of Lemma \ref{lem:hilbert} arises by taking $g=\one$ to be constant:
\begin{equation}\label{eqn:sup/inf-0}
\Theta_+(f,\one) = \log \frac{\sup f}{\inf f}
\quad\Rightarrow\quad
\sup f = e^{\Theta_+(f,\one)} \inf f.
\end{equation}
We will need to relate the Hilbert metric $\Theta_+$ to the supremum norm $\|\cdot\|_\infty$ on $C(X)$. First, we give a short proof of the following useful lemma.

\begin{lemma}[{see \cite[Lemma 2.2]{LSV98}}]\label{lem:T+norm}
Let $X$ be a compact metric space and $m$ a Borel probability measure on $X$.
    Then for every $f,g\in\Lambda$ such that $\int f\, dm=\int g\, dm=1$, we have
\begin{equation}\label{eqn:norm-Theta}
    \|f-g\|_\infty \leq (e^{\Theta_+(f,g)}-1)\min(\|f\|_\infty ,\|g\|_\infty).
\end{equation}
\end{lemma}
\begin{proof}
    Fix $f,g\in C^+$ such that $\int f\,dm=\int g\,dm=1$.   Suppose $\Theta_+(f,g)<\infty$, otherwise there is nothing to prove. Let $A=A(f,g)$ and $B=B(f,g)$ be as in \eqref{eqn:AB}, with the partial order associated to $C^+$, which is just pointwise comparison. Since $\Theta_+(f,g)<\infty$, we have $0 < A \leq B < \infty$. For every $t\in (0,A)$, we have $tf\leq g$. We conclude from this that $Af \leq g$, and similarly, $g \leq Bf$.

By monotonicity of the integral, $A\int f\,dm \leq \int g\,dm\leq B\int f\,dm$, and we obtain $A\leq 1\leq B$. Using nonnegativity of $f$, this implies that
    \[
    (A-B)f\leq (A-1)f\leq g-f\leq (B-1)f\leq (B-A)f,
    \]
    from which we deduce (using \eqref{eqn:Theta}) that
\[
\|g-f\|_\infty \leq (B-A)\|f\|_\infty \leq \frac{B-A}{A}\|f\|_\infty
= (e^{\Theta_+(f,g)} - 1) \|f\|_\infty.
\]
This argument is symmetric in $f$ and $g$, so \eqref{eqn:norm-Theta} is proved.
\end{proof}

\begin{lemma}\label{lem:norm-Theta-2}
Let $X$ be a compact metric space, and $m$ a Borel probability measure on $X$. 
Let $C^+_1(m) := \{ f\in C^+ : \int f\,dm = 1\}$, and suppose that 
$Z \subset C^+_1(m)$ 
and $R,S>0$ are such that
\begin{equation}\label{eqn:RS}
\diam_{\Theta_+}(Z) \leq R
\quad\text{and}\quad
\Theta_+(f,\one) \leq S \text{ for all } f\in Z.
\end{equation}
Then for every $f,g\in Z$, we have
\begin{equation}\label{eqn:norm-Theta-2}
\|g-f\|_\infty \leq R^{-1} (e^R-1) e^S \Theta_+(f,g)
\end{equation}
\end{lemma}
\begin{proof}
Given $f\in Z \subset C^+_1(m)$, we have $\int f\,dm = 1$,
so $\inf f \leq 1 \leq \sup f$, and \eqref{eqn:sup/inf-0} gives
\begin{equation}\label{eqn:norm-leq}
\|f\|_\infty = \sup f = e^{\Theta_+(f,\one)} \inf f \leq e^{\Theta_+(f,\one)} \leq e^S,
\end{equation}
with a similar bound for $\|g\|$.
By Lemma \ref{lem:T+norm}, we have
\begin{equation}\label{eqn:norm-Theta-3}
\|g-f\|_\infty \leq (e^{\Theta_+(f,g)} - 1) e^S.
\end{equation}
Observe that the function $t \mapsto \frac{e^t -1}t$ is increasing on $[0,R]$ since $e^t$ is convex, so for every $t\in [0,R]$, we have
\[
\frac{e^t-1}t \leq \frac{e^R-1}R
\quad\Rightarrow\quad
e^t - 1 \leq R^{-1} (e^R-1) t.
\]
Applying this inequality with $t = \Theta_+(f,g)$ and using \eqref{eqn:norm-Theta-3} proves \eqref{eqn:norm-Theta-2}.
\end{proof}

The following fact is immediate from Definition \ref{def:Hilbert}.

\begin{lemma}\label{lem:nested-cones}
If $\Omega$ is a Banach space and $\Lambda_1 \subset \Lambda_2 \subset \Omega$ are closed convex cones with associated Hilbert metrics $\Theta_1$ and $\Theta_2$, then $\Theta_2(f,g) \leq \Theta_1(f,g)$ for all $f,g \in \Lambda_1$.
\end{lemma}

In particular, given any closed convex cone $\Lambda \subset C^+ \subset C(X)$ with associated Hilbert metric $\Theta$, we have
\begin{equation}\label{eqn:Thetas}
\Theta_+(f,g) \leq \Theta(f,g) \text{ for all } f,g\in \Lambda.
\end{equation}
We see from \eqref{eqn:sup/inf-0} and \eqref{eqn:Thetas} that
if $\one \in \Lambda$, then
\begin{equation}\label{eqn:sup/inf-1}
\sup f \leq e^{\Theta(f,\one)} \inf f
\quad\text{for all } f\in \Lambda.
\end{equation}
If $f,\one\in\Lambda$ and $\int f\,dm=1$ for some $m\in M(X)$, then \eqref{eqn:sup/inf-1} implies
$\|f\|_\infty\leq e^{\Theta(f,\one)}$.

\section{The forward cone RPF theorem}\label{sec:thm1}

In this section, we will prove Theorem \ref{thm:gen-forward}.
Throughout the section, we assume that $\{(X_n,L_n,\Lambda_n)\}_{n\geq 0}$ satisfy \ref{C1}, \ref{C2}, and \ref{C3}.

The proof will be broken up into several lemmas. The first of these 
uses \ref{C2} and the Birkhoff Contraction Theorem \ref{thm:Birkhoff} to deduce exponential cone contraction.

\begin{lemma}\label{lem:cone-contract}
With $\Delta,\TE$ as in \ref{C2}, fix the constants
$\edc = \tanh(\Delta/4)^{1/\TE} \in (0,1)$ and $C_1 = \Delta \edc^{-2\TE}$.
For every $n\in \NN$ and $k\geq\TE$, we have
\begin{equation}\label{eqn:cone-diam}
\diam_{\Theta_{n+k}} \bL_n^k \Lambda_n \leq C_1 \edc^{k+1}.
\end{equation}
\end{lemma}
\begin{proof}
Given $n,k$, let $\ell = \lfloor \frac{k-\TE}{\TE} \rfloor$ so that 
$\ell\TE \leq k-\TE < (\ell+1)\TE$.
Let $t := k-\ell \TE$, so $t\geq \TE$ and \ref{C2} gives $\diam_{\Theta_{n+t}} \bL_n^t \Lambda_n \leq \Delta$. 
For every $i\in \{0,\dots, \ell-1\}$, \ref{C2} and the Birkhoff Contraction Theorem imply that $\bL_{n+t+i\TE}^\TE \colon \Lambda_{n+t+i\TE} \to \Lambda_{n+t+(i+1)\TE}$ contracts the projective metric by a factor of $\tanh(\Delta/4) = \edc^\TE$, and thus
\[
\diam_{\Theta_{n+k}} \bL_n^k\Lambda_n \leq 
(\edc^\TE)^\ell \diam_{\Theta_{n+t}} {\bL_n^t \Lambda_n} \leq
\Delta \edc^{\ell\TE}
\leq \Delta \edc^{k-2\TE+1} = (\Delta \edc^{-2\TE}) \edc^{k+1},
\]
where the last inequality uses the fact that $k-\TE \leq \ell\TE + \TE -1$.
\end{proof}

Now we use Lemma \ref{lem:cone-contract} to prove that the first sequence in \eqref{eqn:lnmn-0} is uniformly exponentially Cauchy, which will imply the claims concerning $\lambda_n$ in \ref{get-lm} and \ref{get-exp1}.

\begin{lemma}\label{lem:lambda}
Let $\edc,C_1$ be as in Lemma \ref{lem:cone-contract}.
Given $n,k\in \NN$ and $\sigma \in M(X_{n+k})$, let
\[
r_{n,k}(\sigma) := \log\frac{\langle \bL_n^k \one, \sigma \rangle}{\langle \bL_{n+1}^{k-1} \one, \sigma \rangle}.
\]
For every $n,k,i,j\in \NN$ satisfying $i,j\geq k$, we have
\begin{equation}\label{eqn:r-leq}
|r_{n,i}(\sigma) - r_{n,j}(\sigma')| \leq C_1 \edc^k
\quad\text{for all } \sigma \in M(X_{n+i})
\text{ and } \sigma' \in M(X_{n+j}).
\end{equation}
\end{lemma}
\begin{proof}
Fix $n,k\in\NN$. Given $i\geq k$ and $\sigma \in M(X_{n+i})$, 
consider the pulled-back measure $\sigma_{k,i} := (\bL_{n+k}^{i-k})^* \sigma \in M(X_{n+k})$.
Then
\[
r_{n,i}(\sigma) = \log \frac
{\langle \bL_n^i \one, \sigma \rangle}
{\langle \bL_{n+1}^{i-1}\one, \sigma \rangle}
= \log\frac 
{\langle \bL_{n+k}^{i-k}(\bL_n^k \one), \sigma \rangle}
{\langle \bL_{n+k}^{i-k}(\bL_{n+1}^{k-1} \one), \sigma \rangle}
= \log\frac
{\langle \bL_n^k \one,\sigma_{k,i} \rangle}
{\langle \bL_{n+1}^{k-1} \one, \sigma_{k,i} \rangle}.
\]
Using this, and a similar expression for $r_{n,j}(\sigma')$ when $j\geq k$ and $\sigma' \in M(X_{n+j})$, we use Lemma \ref{lem:hilbert} to get the following estimate:
\[
|r_{n,i}(\sigma) - r_{n,j}(\sigma')| = \bigg|\log\dfrac{\langle \bL_n^k \one,\sigma_{k,i}\rangle\langle \bL_{n+1}^{k-1}\one,\sigma'_{k,j}\rangle}{\langle \bL_{n+1}^{k-1}\one,\sigma_{k,i}\rangle\langle \bL_n^k \one,\sigma'_{k,j}\rangle}\bigg|
\leq \Theta_{n+k}^+( \bL_n^k\one, \bL_{n+1}^{k-1}\one ).
\]
Since $\one \in \Lambda_n$ by \ref{C1}, we see from \ref{C2} that $L_n\one \in \Lambda_{n+1}$, and thus
both $\bL_n^k\one$ and $\bL_{n+1}^{k-1}\one$ lie in 
$\bL_{n+1}^{k-1} \Lambda_{n+1} \subset \Lambda_{n+k}$. 
By Lemma \ref{lem:nested-cones} and the assumption from \ref{C1} that $\Lambda_{n+k} \subset C_{n+k}^+$, we have $\Theta_{n+k}^+ \leq \Theta_{n+k}$, so
\[
\Theta_{n+k}^+( \bL_n^k\one, \bL_{n+1}^{k-1}\one )
\leq \diam_{\Theta_{n+k}} \bL_{n+1}^{k-1} \Lambda_{n+1}.
\]
Now \eqref{eqn:r-leq} follows from Lemma \ref{lem:cone-contract}.
\end{proof}

We see from Lemma \ref{lem:lambda} that for any $n\geq 0$ and any choice of $\sigma_{n+k}$, the sequence of real numbers $\{ r_{n,k}(\sigma_{n+k})\}_{k\geq 0}$ is Cauchy, and thus $r_n := \lim_{k\to\infty} r_{n,k}(\sigma_{n+k})$ exists. Taking $\lambda_n := e^{r_n}$, we see that the first half of \eqref{eqn:lnmn-0} and of \eqref{eqn:lnmn-exp} is satisfied.\\

To complete the proof of items \ref{get-lm} and \ref{get-exp1}, we turn our attention to the second halves of \eqref{eqn:lnmn-0} and \eqref{eqn:lnmn-exp}, which concern measures on $X_n$.

\begin{lemma}\label{lem:fibEigMeas}
Let $C_1,\edc$ be as in Lemma \ref{lem:cone-contract}.
Given $n,k\in \NN$ and $\sigma \in M(X_{n+k})$, let
\begin{equation}\label{eqn:nu-k}
\nu_{n,k}(\sigma) := \dfrac{({\bL}_n^k)^*\sigma}{\langle\one,({\bL}_n^k)^*\sigma\rangle} \in M(X_n).
\end{equation}
For every $n,k,i,j \in \NN$ satisfying $i,j\geq k$, 
every $\sigma\in M(X_{n+i})$ and $\sigma' \in M(X_{n+j})$,
and every $f\in \Lambda_n$, we have
\begin{equation}\label{eqn:nu-m}
\Big|\int f \, d\nu_{n,i}(\sigma)-\int f  \,d\nu_{n,j}(\sigma') \Big|
\leq C_1\| f \|\edc^k.
\end{equation}
\end{lemma}
\begin{proof}
Given $f\in \Lambda_n$, let $g = f/\|f\|$. Then $\|g\| = 1$, so for every $\nu\in M(X_n)$, we have $\langle g, \nu\rangle \in (0,1]$. Since $\frac d{dt} \log t \geq 1$ on this interval, we see that
\begin{equation}\label{eqn:nu-log}
\begin{aligned}
|\ip{f,\nu_{n,i}(\sigma)} - \ip{f,\nu_{n,j}(\sigma')}|
&= \big|\ip{g,\nu_{n,i}(\sigma)} - \ip{g,\nu_{n,j}(\sigma')}\big| \cdot \|f\| \\
&\leq \big|\log\ip{g,\nu_{n,i}(\sigma)} - \log\ip{g,\nu_{n,j}(\sigma')}\big| \cdot \|f\|.
\end{aligned}
\end{equation}
Given $i\geq k\geq 1$ and $\sigma \in M(X_{n+i})$,
let $\sigma_{k,i} := (\bL_{n+k}^{i-k})^* \sigma \in M(X_{n+k})$.
Then we have
\begin{equation}\label{eqn:f-nui}
\langle g ,\nu_{n,i}(\sigma)\rangle
=\dfrac{\langle g ,( \bL_n^i)^*\sigma\rangle}{\langle\one,( \bL_n^i)^*\sigma\rangle}
=\dfrac{\langle \bL_n^k g ,( \bL_{n+k}^{i-k})^*\sigma\rangle}{\langle \bL_n^k\one,( \bL_{n+k}^{i-k})^*\sigma\rangle}
= \frac{\ip{\bL_n^kg,\sigma_{k,i}}}{\ip{\bL_n^k\one,\sigma_{k,i}}}.
\end{equation}
Combining \eqref{eqn:nu-log} and \eqref{eqn:f-nui} and using Lemmas \ref{lem:hilbert} and \ref{lem:nested-cones} gives
\[
|\langle f, \nu_{n,i}(\sigma)\rangle - \langle f,\nu_{n,j}(\sigma') \rangle|
\leq \Big| \log
\frac{\langle \bL_n^k g, \sigma_{k,i} \rangle \langle \bL_n^k \one, \sigma'_{k,j} \rangle}
{\langle \bL_n^k \one, \sigma_{k,i} \rangle \langle \bL_n^k g, \sigma'_{k,j} \rangle} \Big| \cdot \|f\|
\leq \Theta_{n+k}(\bL_n^k\one, \bL_n^k g) \cdot \|f\|,
\]
and then Lemma \ref{lem:cone-contract} proves \eqref{eqn:nu-m}.
\end{proof}

The following two lemmas will complete the proof of item \ref{get-lm} of Theorem \ref{thm:gen-forward}. The crucial difference from Lemma \ref{lem:fibEigMeas} is that we will use \ref{C3} to consider any $f\in C(X_n)$, not just $f\in \Lambda_n$. The cost is that we will lose the explicit convergence rate in \eqref{eqn:nu-m}.

\begin{lemma}\label{lem:nu-Cauchy}
Fix $n\in \NN$ and $f\in C(X_n)$. For every $\eps>0$, there exists $k\in \NN$ such that for every $i,j\geq k$, every $\sigma\in M(X_{n+i})$, and every $\sigma' \in M(X_{n+j})$, the measures $\nu_{n,i}(\sigma), \nu_{n,j}(\sigma') \in M(X_n)$ defined in \eqref{eqn:nu-k} satisfy 
\begin{equation}\label{eqn:sigma'}
\Big|\int f \, d\nu_{n,i}(\sigma)-\int f  \,d\nu_{n,j}(\sigma')\Big| < \eps.
\end{equation}
\end{lemma}
\begin{proof}
First observe that for any $p\in \NN$ and any $g\in \Lambda_n^p$, we have $\bL_n^p g \in \Lambda_{n+p}$ by definition, and thus for every $i,j \geq k\geq p$, \eqref{eqn:nu-m} gives
\begin{equation}\label{eqn:nu-p}
\Big|\int g \, d\nu_{n,i}(\sigma)-\int g  \,d\nu_{n,j}(\sigma')\Big|\leq C_1\| \bL_n^p g \|\edc^{k-p}.
\end{equation}
Given any $f\in C(X)$ and any $\eps>0$, by \ref{C3}, there exist $p\in \NN$ and $f^\pm \in \Lambda_n^p$ such that $\|f - (f^+ - f^-)\| < \eps/4$. 
Since $\edc\in (0,1)$, there exists $k\in \NN$ such that
$C_1 \|\bL_n^p g\| \edc^{k-p} < \eps/4$ for $g=f^+$ and $g=f^-$.
Then for every $i,j\geq k$, we have
\begin{align*}
| \ip{f,\nu_{n,i}(\sigma)} - \ip{f,\nu_{n,j}(\sigma')}|
&\leq |\ip{(f^+ - f^-),\nu_{n,i}(\sigma)} - \ip{(f^+ - f^-),\nu_{n,j}(\sigma')}| + 2\eps/4 \\
&\leq C_1 (\|\bL_n^p f^+ \| + \|\bL_n^p f^-\|) \edc^{k-p} + \eps/2 < \eps.\qedhere
\end{align*}
\end{proof}

\begin{lemma}\label{lem:mn}
For every $n\in \NN$, there exists $m_n \in M(X_n)$ such that given any $\{\sigma_{n+k} \in M(X_{n+k}) \}_{k\in \NN}$, the sequence $\nu_{n,k}(\sigma_{n+k}) \in M(X_n)$ converges to $m_n$ in the weak* topology as $k\to\infty$. Moreover, this limit is independent of the choice of $\sigma_{n+k}$.
\end{lemma}
\begin{proof}
For every $f\in C(X_n)$,
the sequence $\{\ip{f,\nu_{n,k}(\sigma_{n+k})}\}_{k\in\NN}$ is Cauchy
by Lemma \ref{lem:nu-Cauchy}. In particular, the following limit exists:
\begin{equation}\label{eqn:ell-n}
\ell_n(f) := \lim_{k\to\infty} \int f \,d\nu_{n,k}(\sigma_{n+k}).
\end{equation}
The map $\ell_n \colon C(X_n) \to \RR$ is a pointwise limit of positive uniformly bounded linear functionals, so it is a positive bounded linear functional, and hence by the Riesz representation theorem there exists $m_n \in M(X_n)$ such that $\ell_n(f) = \int f\,dm_n$ for all $f\in C(X)$. Then \eqref{eqn:ell-n} implies that $\nu_{n,k}(\sigma_{n+k})\to m_n$ in the weak* topology. The fact that the limit is independent of the choice of $\sigma_{n+k}$ follows from the fact that in Lemma \ref{lem:nu-Cauchy}, the choice of $k=k(\eps)$ is independent of the measures $\sigma,\sigma'$.
\end{proof}

At this point we can also observe that \eqref{eqn:nu-m}
 establishes item \ref{get-exp1} of Theorem \ref{thm:gen-forward}, so it remains to prove items \ref{get-L*} and \ref{get-m!}, which is done in the next lemma.

\begin{lemma}\label{lem:eigen1}
  The measures $\{m_n\}$ given by Lemma \ref{lem:mn} 
satisfy $L_n^* m_{n+1} = \lambda_n m_n$ for every $n\in \NN$, and
are the unique Borel probability measures for which ${L}_n^*m_{n+1}$ is a scalar multiple of $m_n$ for every $n$.  
\end{lemma}

\begin{proof}
    First, we show that ${L}_n^*m_{n+1}=\lambda_nm_n$. For all $ f \in C(X_n)$, we have 
    \begin{align*}
        \langle f , L_n^*m_{n+1}\rangle
&= \ip{ L_n f, m_{n+1} }
=\lim_{k\to\infty}\dfrac{\langle L_n f ,( \bL^k_{n+1})^*\sigma_{n+1+k}\rangle}{\langle\one,( \bL^k_{n+1})^*\sigma_{n+1+k}\rangle}\\        &=\lim_{k\to\infty}\dfrac{\langle \bL^k_{n+1}L_n f ,\sigma_{n+1+k}\rangle}{\langle \bL^k_{n+1}\one,\sigma_{n+1+k}\rangle}\cdot\dfrac{\langle \bL_n^{k+1}\one,\sigma_{n+1+k}\rangle}{\langle \bL^{k+1}_{n}\one,\sigma_{n+1+k}\rangle}\\        &=\lim_{k\to\infty}\dfrac{\langle \bL_n^{k+1}\one,\sigma_{n+1+k}\rangle}{\langle \bL^k_{n+1}\one,\sigma_{n+1+k}\rangle}\cdot\dfrac{\langle \bL_n^{k+1} f ,\sigma_{n+1+k}\rangle}{\langle \bL^{k+1}_{n}\one,\sigma_{n+1+k}\rangle}
        =\lambda_n {\langle f ,m_n}\rangle,
    \end{align*}
where the last equality uses Lemmas \ref{lem:lambda} and \ref{lem:mn} (sending $k+1\to\infty$).

Now suppose that for some sequence of probability measures $\nu_n\in M(X_n)$ and some $\xi_n \in (0,\infty)$, we have $(L_n)^*\nu_{n+1}=\xi_n\nu_n$ for all $n\in\NN$. 
To show that $\nu_n = m_n$, we recall from Lemma \ref{lem:mn} that the definition of $m_n$ is independent of which sequence of measures we use, and thus for any $ f \in C(X_n)$, we have
\[
\langle f ,m_n\rangle
=\lim_{k\to\infty}\dfrac{\langle \bL_n^k f ,\nu_{n+k}\rangle}{\langle \bL_n^k\one,\nu_{n+k}\rangle}
=\lim_{k\to\infty}\dfrac{\langle  f ,\xi_{n+k-1}\cdots\xi_n\nu_{n}\rangle}{\langle \one,\xi_{n+k-1}\cdots\xi_n\nu_{n}\rangle}
=\langle f ,\nu_n\rangle.
\]
Since the measures determine the scaling factors,
this completes the proof.
\end{proof}


\section{A bi-infinite nonstationary RPF theorem}
\label{sec:thm2}

In this section, we will prove Theorem \ref{thm:gen-two-sided}. Throughout the section, we assume that $\{(X_n,L_n,\Lambda_n)\}_{n\in\ZZ}$ satisfy \ref{C1}, \ref{C2}, and \ref{C3}.

Applying Theorem \ref{thm:gen-forward} to each forward-infinite part of this sequence, we obtain $\lambda_n \in (0,\infty)$ and $m_n \in M(X_n)$ for every $n\in \ZZ$ that satisfy \ref{get-lm}--\ref{get-exp1}. Using $\lambda_n$ and $m_n$, we make the following definitions for each $n\in \ZZ$:
\begin{align*}
\Lambda_n^1 := \big\{ f \in \Lambda_n : \ip{f,m_n} = 1 \big\}
\quad\text{ and }\quad
\widehat{L}_n := \lambda_n^{-1}L_n \colon \Lambda_n \to \Lambda_{n+1}.
\end{align*}

\begin{lemma}\label{lem:Ln1}
For every $f\in \Lambda_n^1$, we have $\widehat{L}_n f \in \Lambda_{n+1}^1$.
\end{lemma}
\begin{proof}
It follows from the definitions that
\[
\ip{\widehat{L}_n f, m_{n+1}}
= \lambda_n^{-1} \ip{L_n f, m_{n+1}}
= \lambda_n^{-1} \ip{f, L_n^* m_{n+1}}
= \lambda_n^{-1} \ip{f, \lambda_n m_n}
= \ip{f, m_n} = 1,
\]
which proves the lemma.
\end{proof}


\begin{lemma}\label{lem:get-h-1}
Let $C_1>0$ and $\edc \in (0,1)$ be given by Lemma \ref{lem:cone-contract}, and let $\Delta>0$ and $\TE\in\NN$ be as in \ref{C2}. 
Write $C_2 := C_1\Delta^{-1} e^{2\Delta}(e^\Delta -1)$. Given any $n\in \ZZ$ and $i,j\geq k\geq \TE$, we have
\begin{equation}\label{eqn:norm-close}
\|\widehat{\bL}_{n-i}^i f - \widehat{\bL}_{n-j}^j g \| \leq C_2 \edc^k
\quad\text{for all } f\in \Lambda_{n-i}^1
\text{ and } g\in \Lambda_{n-j}^1.
\end{equation}
Moreover, for each $i\geq \TE$ and $f\in \Lambda_{n-i}^1$, we have $\|\log \widehat{\bL}_{n-i}^i f\| \leq 2\Delta$.
\end{lemma}
\begin{proof}
Observe that $\widehat{\bL}_{n-i}^i f$ and $\widehat{\bL}_{n-j}^j g$ both lie in $\bL_{n-k}^k \Lambda_{n-k} \cap \Lambda_n^1$; indeed, they are contained in the first set since $i,j\geq k$, and in the second set by Lemma \ref{lem:Ln1}. 
We will apply Lemma \ref{lem:norm-Theta-2} with $Z = \bL_{n-k}^k \Lambda_{n-k} \cap \Lambda_n^1$. To do so, we must determine appropriate values of $R,S$ satisfying \eqref{eqn:RS}.

Since $k\geq \TE$, we see that $\diam_{\Theta_n} Z \leq \Delta$ by \ref{C2}, so we take $R=\Delta$. 
We also see from \ref{C2} and \eqref{eqn:sup/inf-0} that
\[
\Theta_+(\bL_{n-\TE}^\TE\one, \one) = \log \frac{\sup \bL_{n-\TE}^\TE\one}{\inf \bL_{n-\TE}^\TE\one} \leq \Delta,
\]
and so to determine an appropriate value of $S$, we observe that given any $h\in Z$, we have
\begin{equation}\label{eqn:h-1}
\Theta_+(h,\one)
\leq \Theta_+(h, \bL_{n-\TE}^\TE \one) + \Theta_+(\bL_{n-\TE}^\TE\one, \one)
\leq 2\Delta =: S,
\end{equation}
Now we can apply Lemma \ref{lem:norm-Theta-2} and obtain
\begin{equation}\label{eqn:hL-1}
\|\widehat{\bL}_{n-i}^i f - \widehat{\bL}_{n-j}^j g \|
\leq \Delta^{-1} (e^\Delta - 1) (e^{2\Delta}) \Theta_n(\widehat{\bL}_{n-i}^i f, \widehat{\bL}_{n-j}^j g).
\end{equation}
Moreover, Lemma \ref{lem:cone-contract} gives
\[
\Theta_n(\widehat{\bL}_{n-i}^i f, \widehat{\bL}_{n-j}^j g)
\leq \diam_{\Theta_n} \bL_{n-k}^k \Lambda_{n-k}
\leq C_1 \edc^k,
\]
and combining this with \eqref{eqn:hL-1} proves \eqref{eqn:norm-close}.

For the last claim in the lemma, we observe that $\widehat{\bL}_{n-i}^i f \in Z$, and every $h\in Z$ satisfies $\sup h \leq e^{2\Delta} \inf h$ by \eqref{eqn:h-1} and \eqref{eqn:sup/inf-0}; moreover, since $\int h \,dm_n = 1$, we have $\inf h \leq 1\leq \sup h$, so $h(x) \in [e^{-2\Delta},e^{2\Delta}]$ for all $x\in X_n$, which completes the proof.
\end{proof}

Now we can prove \ref{get-h} and \ref{get-exp2}. Given any sequence $\{ f_\ell \in \Lambda_\ell\}_{\ell \leq n}$, consider the normalized sequence $g_\ell = f_\ell / \ip{f_\ell, m_\ell}$, which satisfies $g_\ell \in \Lambda_\ell^1$. 
Then the sequence of functions on the right-hand side of \eqref{eqn:hn}
can be rewritten as
\[
\frac{\bL_{n-k}^k f_{n-k}}{\bl_{n-k}^k \ip{f_{n-k},m_{n-k}}}
= \widehat{\bL}_{n-k}^k g_{n-k}.
\]
This sequence is uniformly Cauchy by Lemma \ref{lem:get-h-1}, and thus converges uniformly to some $h_n \in C(X_n)$. The fact that the limit is independent of the choice of $f_\ell$ follows from the fact that the upper bound in \eqref{eqn:norm-close} does not depend on the choice of $f$ or $g$. The exponential convergence in \eqref{eqn:hn-exp} follows immediately from \eqref{eqn:norm-close}.

To prove \ref{get-Lh}, first observe that
$h_n \in \Lambda_n \cup \{0\}$ since this set is closed in the uniform norm by \ref{C1}. Since each $\widehat{\bL}_{n-k}^k g_{n-k}$ lies in $\Lambda_n^1$, we have
\begin{equation}\label{eqn:int-hn}
\int h_n \,dm_n = \lim_{k\to\infty} \int \widehat{\bL}_{n-k}^k g_{n-k} \,dm_n = 1,
\end{equation}
so $h_n \in \Lambda_n$. Strict positivity of $h_n$ follows from the last inequality in Lemma \ref{lem:get-h-1}, which in fact yields the following explicit bounds:
\begin{equation}\label{eqn:h-range}
e^{-2\Delta} \leq h_n(x) \leq e^{2\Delta}
\quad\text{for all } x\in X_n.
\end{equation}
To prove $L_n h_n = \lambda_n h_{n+1}$, 
fix $\{ f_\ell \in \Lambda_\ell^1 \}_{\ell\leq n}$ and
observe that since $L_n$ (and hence $\widehat{L}_n$) is bounded, we have
\[
\widehat{L}_n h_n = \lim_{k\to\infty} \widehat{L}_n \widehat{\bL}_{n-k}^k f_{n-k}
= \lim_{k\to\infty} \widehat{\bL}_{n+1-(k+1)}^{k+1} f_{n-k} = h_{n+1}.
\]
It only remains to prove the uniqueness claim in \ref{get-h!}. Suppose that $\{g_n \in \Lambda_n\}_{n\in \ZZ}$ and $\{ \xi_n \in \RR \}_{n\in \ZZ}$ are such that 
$\int g_n \,dm_n = 1$ and $L_n g_n = \xi_n g_{n+1}$ for all $n\in \ZZ$. 
Then we have
\[
\xi_n = \ip{\xi_n g_{n+1}, m_{n+1}} = \ip{L_n g_n, m_{n+1}}
= \ip{g_n, L_n^* m_{n+1}} = \ip{ g_n, \lambda_n m_n } = \lambda_n
\]
for every $n$, and thus for every $n\in \ZZ$ and $k\geq 0$, we have
\[
\bL_{n-k}^k g_{n-k} = \Big(\prod_{i=n-k}^{n-1} \lambda_i\Big) g_n
= \bl_{n-k}^k g_n
\quad\Rightarrow\quad
\widehat{\bL}_{n-k}^k g_{n-k} = g_n.
\]
As $k\to\infty$, the left-hand side converges to $h_n$ by the fact that the limit in \eqref{eqn:hn} is independent of the choice of starting functions, and we conclude that $g_n = h_n$. This completes the proof of Theorem \ref{thm:gen-two-sided}.

\section{Convergence to Pseudo-invariant Measures}\label{sec:convergence-pseudo}

In this section, we prove Theorem \ref{thm:pseudo-inv}. Unlike in \S\ref{sec:thm1} and \S\ref{sec:thm2}, here it is important that our operators $L_n$ are the RPF operators associated to maps $T_n$ and potentials $\ph_n$. In particular, given any $f\in C(X_{n+1})$ and $g\in C(X_n)$, we have the following for each $x\in X_n$:
\[
L_n (g \cdot (f\circ T_n))(x)
= \sum_{y\in T_n^{-1}(x)} e^{\ph_n(y)} (g \cdot (f\circ T_n))(y)
= f(x) (L_n g)(x).
\]
This implies that
\begin{equation}\label{eqn:gfT}
\ip {g\cdot (f\circ T_n), m_n}
= \lambda_n^{-1} \ip{g\cdot (f\circ T_n), L_n^* m_{n+1}}
= \lambda_n^{-1} \ip{f \cdot L_n g, m_{n+1}}.
\end{equation}
Taking $g = h_n$, \eqref{eqn:gfT} gives the following for every $f\in C(X_{n+1})$:
\[
\int_{X_n} (f\circ T_n) h_n \,dm_n = \lambda_n^{-1} \int_{X_{n+1}} f \cdot (L_n h_n) \,dm_{n+1}
= \int_{X_{n+1}} f h_{n+1} \,dm_{n+1}.
\]
This implies that for $\mu_n-h_nm_n$, we have
\[
\int_{X_{n+1}} f\,d((T_n)_* \mu_n) = \int_{X_n} (f\circ T_n) \,d\mu_n = \int_{X_{n+1}} f\,d\mu_{n+1}.
\]
Hence, we conclude that $(T_n)_* \mu_n = \mu_{n+1}$ as claimed.

To prove the claims in \eqref{eqn:tilde-L}, first recall the normalized operator $\widetilde{L}_n$ for the potential $\Tilde{\varphi}=\varphi_n+\log h_n-\log h_{n+1}\circ T_n-\log\lambda_n$. Observe that given $f\in C(X_n)$ and $x\in X_{n+1}$, we have
\begin{equation}\label{eqn:tilde-Ln}
(\widetilde{L}_n f)(x)
= \sum_{y\in T_n^{-1} x} e^{\widetilde\ph_n(y)} f(y)
= \sum_{y\in T_n^{-1} x} \frac{e^{\ph_n(x)} h_n(y)}{h_{n+1}(T_n y) \lambda_n} f(y)
= \frac {L_n(h_n f)(x)}{\lambda_n h_{n+1}(x)} .
\end{equation}
Taking $f=\one$, this gives
\[
\widetilde{L}_n \one = \frac{L_n h_n}{\lambda_n h_{n+1}} = \one,
\]
which proves the first part of \eqref{eqn:tilde-L}. Moreover, for every $f\in C(X_n)$, \eqref{eqn:tilde-Ln} gives
\begin{align*}
\ip{f, \widetilde{L}_n^*\mu_{n+1}} &= \ip{\widetilde{L}_n f, \mu_{n+1}}
= \int_{X_{n+1}} \frac{L_n(h_n f)}{\lambda_n h_{n+1}} \,d\mu_{n+1}
= \ip{L_n(h_n f), \lambda_n^{-1} m_{n+1}} \\
&= \ip{h_n f, \lambda_n^{-1} L_n^* m_{n+1}} = \ip{h_n f, m_n}
= \int_{X_n} f h_n \,dm_n = \ip{f, \mu_n},
\end{align*}
and thus $\widetilde{L}_n^* \mu_{n+1} = \mu_n$, proving the second part of \eqref{eqn:tilde-L}. This finishes the proof of Theorem \ref{thm:pseudo-inv}.

\section{H\"older dependence}
\label{sec:Holder}

In this section we prove Theorems \ref{thm:HolDep-intro} and \ref{thm:continuity} on H\"older dependence of $\lambda_n,m_n,h_n$. We start in \S\ref{sec:op-dep} with the proof of Theorem \ref{thm:continuity}, showing how these objects depend H\"older continuously on the operators. Then in \S\ref{sec:sys-dep}, we show that the operators depend H\"older continuously on $((T_n,\ph_n))_n$, proving Theorem \ref{thm:HolDep-intro}.

\subsection{Dependence on operators}\label{sec:op-dep}

Now we fix cones $\Lambda_n \subset C_n^+ \subset C(X_n)$ that satisfy \ref{C1}, and we fix parameters $\kappa\geq 1$, $\Delta>0$, and $\tau\in \NN$. Recall that $\LL_0^+$ and $\LL_0$ denote the spaces of one- and two-sided sequences of operators $\bL$ such that \ref{C2}--\ref{C3} hold with this choice of $\Lambda_n,\kappa,\Delta,\tau$.
Recall also that we define a metric on $\LL_0$ in \eqref{eqn:dL} by
\[
d_\LL(\bL,\bL') := \sum_{n\in \ZZ} 2^{-|n|} \|L_n - L_n'\|_{\Lambda_n},
\]
and similarly for $d_{\LL^+}$ with $\sum_{n\geq 0}$.
Theorem \ref{thm:continuity} claims the following results with respect to 
these metrics:
\begin{enumerate}[label=\textup{(\alph{*})},leftmargin=*]
\item The map $\bL\in \LL_0^+ \mapsto \lambda_0\in (0,\infty)$ is locally H\"older continuous.
\item For all $f \in \Lambda_0$, the map $\bL\in \LL_0^+ \mapsto \int f \,dm_0\in \RR$ is locally H\"older continuous.
\item The map $\bL\in \LL_0 \mapsto h_0\in C(X_0)$ is locally H\"older continuous with respect to the uniform norm on $C(X_0)$.
\end{enumerate}

To prove these claims, we will need the following auxiliary lemmas.

\begin{lemma}\label{lem:diff L}
For all $n\in\ZZ$, $j\in\NN$, $f\in\Lambda_n$, and $\bL,\bL'\in \LL_0^+$, we have
\begin{equation}\label{eqn:Lnjf}
\|\bL_n^j f-(\bL')_n^j f\|_\infty\leq 2^{|n|}(2\kappa)^{j}\|f\|_{\infty}d_\LL(\bL,\bL').
\end{equation}
\end{lemma}
\begin{proof}
Fix $n\in\ZZ$. 
We will prove \eqref{eqn:Lnjf} by induction in $j$.
By the definition of $d_\LL$, we have
\begin{equation}\label{eqn:LndL}
\|L_n f-(L')_n f\|_\infty\leq 2^{|n|}\|f\|_{\infty}d_\LL(\bL,\bL').
\end{equation}
Since $1\leq 2\kappa$, this proves \eqref{eqn:Lnjf} for $j=1$. Now suppose that \eqref{eqn:Lnjf} holds for some $j\geq 1$. Then we have
    \begin{align*}
        \|\bL_{n}^{j+1} f-(\bL')_{n}^{j+1} f\|_\infty
        &\leq \|L_{n+j}(\bL_{n}^{j} f)-L_{n+j}((\bL')_{n}^{j} f)\|_\infty\\
        &\qquad\qquad\qquad +\|L_{n+j}((\bL')_{n}^{j} f)-(L')_{n+j}((\bL')_{n}^{j} f)\|_\infty
\end{align*}
Since $\|L_{n+j}\|_{\Lambda_{n+j}} \leq \kappa$ and $\bL_n^j f \in \Lambda_{n+j}$, the inductive hypothesis gives
\begin{align*}
\|L_{n+j}(\bL_{n}^{j} f)-L_{n+j}((\bL')_{n}^{j} f)\|_\infty
&\leq 
\|L_{n+j}\|_{\Lambda_{n+j}} \|\bL_{n}^{j} f-(\bL')_{n}^{j} f\|_\infty \\
&\leq \kappa 2^{|n|} (2\kappa)^j \|f\|_\infty d_\LL(\bL,\bL').
\end{align*}
Similarly, since $\|(\bL')_n^j\|_{\Lambda_n} \leq \kappa^j$, we can use \eqref{eqn:LndL} to get
\begin{align*}
\|L_{n+j}((\bL')_{n}^{j} f)-(L')_{n+j}((\bL')_{n}^{j} f)\|_\infty
&\leq 2^{|n+j|} \|(\bL')_n^j f\|_\infty d_\LL(\bL,\bL') \\
&\leq (2\kappa)^j 2^{|n|} d_\LL(\bL,\bL').
\end{align*}
Combining these estimates, we obtain
\[
\|\bL_{n}^{j+1} f-(\bL')_{n}^{j+1} f\|_\infty
\leq (\kappa + 1) 2^{|n|} (2\kappa)^j \|f\|_\infty d_\LL(\bL,\bL'),
\]
which completes the inductive step and proves the lemma since $\kappa + 1 \leq 2\kappa$.
\end{proof}

\begin{lemma}\label{lem:comparable}
For any $0<\gamma<1<\theta$ and constants $D_1,D_2>0$,
there exist $C>0$ and $\alpha \in (0,1)$ such that for every $k\in \NN$ and every $0 < r < (\gamma/\theta)^k$, there exists $j\in \NN$ such that $j\geq k$ and 
\[
D_1\gamma^j+D_2\theta^j r\leq C r^\alpha.
\]
\end{lemma}
\begin{proof}
    Note that $0<\gamma<1<\theta$ implies $\rho=\gamma/\theta<1$ and $0<\alpha:=\log\gamma/\log\rho<1$. Let $j\in\NN$ be such that
    \[
    \rho^{j+1}\leq r
    \leq \rho^j.
    \]
Since $\rho^{j+1} \leq r < \rho^k$ by assumption, we have $j+1 > k$, so $j\geq k$.
    For this choice of $j$, we have
    \[
    \gamma^j= e^{j\log\gamma}= e^{\alpha j\log\rho}= \rho^{\alpha j} \leq \rho^{-\alpha} r^\alpha.
    \]
    Thus, $\theta^j r\leq \gamma^j\leq\rho^{-\alpha} r^\alpha$. The result follows with $C:=\rho^{-\alpha}\max\{D_1,D_2\}$.
\end{proof}

\vskip5pt
In the ``locally H\"older continuous'' arguments that follow, we will always work with $\bL$, $\bL'$ such that $d_\LL(\bL,\bL') < (\gamma/\theta)^\tau$ for an appropriate choice of $\theta$, so that we can apply Lemma \ref{lem:comparable} with $j\geq \tau$.

Now we prove Theorem \ref{thm:continuity} via a sequence of three lemmas.

\begin{lemma}\label{lem:eigenvalue}
   The map $\LL_0^+ \to (0,\infty)$ taking $\bL \mapsto \lambda_0$ is locally H\"older continuous.
\end{lemma}
\begin{proof}
It suffices to show that $r_0=\log\lambda_0$ is locally H\"older in $\bL$. 

Given $\bL,\bL' \in \LL_0^+$
and $j\in \NN$, fix a probability measure $\sigma\in{M}(X_{j})$.
Let 
\begin{equation}\label{eqn:r0j'}
r_{0,j}=\log\frac{\langle \bL_0^j \one, \sigma \rangle}{\langle \bL_{1}^{j-1} \one, \sigma \rangle}
\quad\text{and}\quad r_{0,j}'=\log\frac{\langle (\bL')_0^j \one, \sigma \rangle}{\langle (\bL')_{1}^{j-1} \one, \sigma \rangle}.
\end{equation}
By Lemma \ref{lem:lambda}, there exists $C_1>0$ and $\gamma\in(0,1)$, depending only on $\Lambda_n,\kappa,\Delta,\tau$, such that for all $j\geq\tau$, we have
\begin{equation}\label{eqn:r0j}
|r_{0,j}-r_0|\leq C_1\gamma^j \quad \text{ and } \quad |r_{0,j}'-r'_0|\leq C_1\gamma^j.
\end{equation}

To bound $|r_{0,j}-r_{0,j}'|$, we first observe that by \ref{C2}, we have
\[
\min\big( 
\langle \bL_0^j \one, \sigma \rangle,\ 
\langle \bL_{1}^{j-1} \one, \sigma \rangle,\
\langle (\bL')_0^j \one, \sigma \rangle,\ 
\langle (\bL')_{1}^{j-1} \one, \sigma \rangle
\big) \geq \kappa^{-j},
\]
and that as in \eqref{eqn:logrs}, we have
\[
|\log s - \log t| \leq \kappa^j |s-t|
\quad\text{for all } s,t\geq \kappa^{-j}.
\]
Using these together with the definitions in \eqref{eqn:r0j'} gives
\begin{align*}
        |r_{0,j}-r_{0,j}'|
        &\leq \big|\log\langle \bL_0^j\one,\sigma\rangle-\log\langle (\bL')_0^j\one,\sigma\rangle\big|\\ &\qquad\qquad+\big|\log\langle \bL_1^{j-1}\one,\sigma\rangle-\log\langle (\bL')_1^{j-1}\one,\sigma\rangle\big|\\
        &\leq \kappa^j \big(\|\bL_0^j \one-(\bL')_0^j \one\|_\infty+\|\bL_1^{j-1} \one-(\bL')_1^{j-1} \one\|_\infty\big).
\end{align*}
Applying Lemma \ref{lem:diff L} to each of the terms in this expression, we get
\begin{align*}
|r_{0,j} - r_{0,j}'|
&\leq \kappa^j \big( (2\kappa)^j d_\LL(\bL,\bL') + 2 (2\kappa)^{j-1} d_\LL(\bL,\bL') \big) \\
&\leq 2 \kappa^j (2\kappa)^j d_\LL(\bL,\bL').
\end{align*}

Writing $\theta := 2\kappa^2$ and recalling \eqref{eqn:r0j}, we have
\[
|r_{0}-r_0'| \leq 2C_1\gamma^j +2 \theta^j d_\LL(\bL,\bL'). 
\]
Now we can apply Lemma \ref{lem:comparable} with $D_1 = 2C_1$ and $D_2 = 2$ to get $C>0$ and $\alpha \in (0,1)$ such that for every $\bL,\bL' \in \LL_0^+$ satisfying $d_\LL(\bL,\bL') < (\gamma/\theta)^\tau$, we have
$|r_0 - r_0'| \leq C d_\LL(\bL,\bL')^\alpha$, which proves the lemma.
\end{proof}

\begin{lemma}
\label{lem:mHolder}
   For all $f \in \Lambda_0$, the map $\LL_0^+ \to \RR$ taking $\bL \mapsto \int f \,dm_0$ is locally H\"older continuous.
\end{lemma}
\begin{proof}
Let $\theta := 2\kappa^4$, and fix $\bL,\bL' \in \LL_0^+$ such that $d_\LL(\bL,\bL') < (\gamma/\theta)^\tau$.
For convenience, we write $F_j={\langle \bL_0^j f ,\sigma\rangle}/{\langle \bL_0^j\one,\sigma\rangle}$ and similarly for $F_j'$. So by Lemma \ref{lem:fibEigMeas}, $F_j \to \langle f,m_0\rangle$ and $F_j' \to \langle f,m_0'\rangle$.

Fix $j\geq \tau$. Let $\sigma\in M(X_j)$ and $f\in\Lambda_0$. Then
\begin{equation}\label{eqn:fm0}
\begin{aligned}
        \big|\langle f,m_0\rangle-\langle f,m_0'\rangle\big|
        &\leq \big|\langle f,m_0\rangle-F_j\big|
        + \big|F_j-F_j'\big| 
        + \big|F_j'-\langle f,m_0'\rangle\big|\\
        &\leq 2C_1\gamma^j
        +\big|F_j-F_j'\big|
    \end{aligned}
\end{equation}
    where $C_1>0$ and $\gamma>1$ are as in Lemma \ref{lem:lambda}. (In particular, they do not depend on $\bL$, $\bL'$, or $j$.)

    Now we bound the quantity $|F_j - F_j'|$. By conditions \ref{C1} and \ref{C2}, we have $\Theta^+(\bL_0^\tau f, \one) \leq \Delta$, so $e^{-\Delta} \leq \bL_0^\tau f(x) \leq e^\Delta$ for all $x\in X_\tau$, and consequently
\[
\kappa^{-j} e^{-\Delta} \leq
\kappa^{-(j-\Delta)} e^{-\Delta} \leq \|\bL_0^j f\|_\infty \leq \kappa^{j-\Delta} e^\Delta
\leq \kappa^j e^\Delta.
\]
A similar bound holds for $(\bL')_0^j f$, and thus we have
\begin{gather*}
\min\big( 
\langle \bL_0^j f, \sigma \rangle,\ 
\langle \bL_0^j \one, \sigma \rangle,\
\langle (\bL')_0^j f, \sigma \rangle,\ 
\langle (\bL')_{0}^{j} \one, \sigma \rangle
\big) \geq \kappa^{-j} e^{-\Delta}, \\
\min\big( F_j,\ F_j'\big) \geq \kappa^{-2j} e^{-2\Delta}.
\end{gather*}
Now we can once again use the bound from \eqref{eqn:logrs} to obtain
\[
|F_j-F_j'|\leq e^{2\Delta} \kappa^{2j} |\log F_j-\log F_j'|.
\]
Another application of \eqref{eqn:logrs} and Lemma \ref{lem:diff L} gives
\begin{align*}
        \big|\log F_j-\log F_j'\big|
        &\leq \big|\log\langle \bL_0^jf,\sigma\rangle-\log\langle (\bL')_0^jf,\sigma\rangle\big|\\ 
&\qquad\qquad+\big|\log\langle \bL_0^{j}\one,\sigma\rangle-\log\langle (\bL')_0^{j}\one,\sigma\rangle\big|\\
        &\leq e^{\Delta} \kappa^j \big(\|\bL_0^j f-(\bL')_0^j f\|_{\infty}+\|\bL_0^{j} \one-(\bL')_0^{j} \one\|_\infty\big)\\
        &\leq e^{\Delta} \kappa^j (1+\|f\|_{\infty}) (2\kappa)^{j}d_\LL(\bL,\bL').
\end{align*}
Combining the previous two estimates gives
\[
|F_j - F_j'| \leq e^{3\Delta} 
(1+\|f\|_{\infty}) (2\kappa^4)^{j}d_\LL(\bL,\bL').
\]

Together with \eqref{eqn:fm0}, we see that for all $j\geq \tau$, we have
\[
\big|\langle f,m_0\rangle-\langle f,m_0'\rangle\big|
    \leq 2C_1 \gamma^j
    + e^{3\Delta} (1+\|f\|_{\infty}) (2\kappa^4)^j d_\LL(\bL,\bL').
\]
Now we can apply Lemma \ref{lem:comparable} with $\theta = 2\kappa^4$, $D_1 = 2C_1$, and $D_2 = e^{3\Delta}(1+\|f\|_\infty)$. Observe that since we assumed $d_\LL(\bL,\bL') < (\gamma/\theta)^\tau$, the $j$ produced by Lemma \ref{lem:comparable} satisfies $j\geq \tau$, and we obtain the  desired result.
\end{proof}

\begin{lemma}
    The map $\LL_0 \to C(X_0)$ taking $\bL \mapsto h_0$ is locally H\"older continuous with respect to the uniform norm on $C(X_0)$.
\end{lemma}
\begin{proof}
By Lemma \ref{lem:get-h-1}, $\widehat{\bL}_{-j}^j \one$ converges uniformly to $h_0$ with exponential speed: there exist $C_2>0$ and $\gamma\in (0,1)$, independent of $\bL,\bL'$, such that we have
    \begin{align*}
        \|h_0-h_0'\|_\infty
        &\leq \big\|h_0-\widehat{\bL}_{-j}^j \one\big\|_\infty + \big\|\widehat{\bL}_{-j}^j \one-\widehat{(\bL')}_{-j}^j \one\big\|_\infty + \big\|\widehat{(\bL')}_{-j}^j \one-h_0'\big\|_\infty\\
        &\leq 2C_2 \gamma^j + 
2^{|-j|} (2\kappa)^j d_\LL(\widehat{\bL},\widehat{\bL'}),
    \end{align*}
for all $j\in \NN$, where the last inequality also uses Lemma \ref{lem:diff L}. 
    Applying Lemma \ref{lem:comparable} with $D_1 = 2C_2$, $D_2 = 1$, and $\theta = 4\kappa$ gives the desired result.
\end{proof}

\subsection{Dependence on maps and potentials}\label{sec:sys-dep}

Now we turn our attention to the case when the operator sequence $\bL$ is determined by a sequence $\{(X_n,T_n,\ph_n)\}_n$ as in \eqref{eqn:Ln}, and prove Theorem \ref{thm:HolDep-intro}. Recall that we fix a sequence of compact metric spaces $\{(X_n,d_n)\}_n$, together with constants $D,\tau\in \NN$, $\rho,\beta\in (0,1)$, and $\delta,V,H>0$, and let $\mathcal{F}$ denote the collection of sequences of map-potential pairs $(\bT,\bph)=\{(T_n,\ph_n)\}_{n\in\ZZ}$ satisfying \allA\ for these choices of parameters, and similarly for $\mathcal{F}^+$. The arguments in Section \ref{sec:get-cones} show that the sequence of operators
\[
\bL(\bT,\bph)_n(f(x))=\sum_{y\in T_n^{-1}x}e^{\ph_n(y)}f(y)
\]
 satisfies conditions \ref{C1}--\ref{C3} with cones $\Lambda_n$ and constants $\kappa,\Delta,\tau$ that depend only on $D,\tau,\rho,\delta,V,H,\beta$, and not on the specific choice of $(\bT,\bph)$. Thus $(\bT,\bph) \mapsto \bL(\bT,\bph)$ is a function from $\mathcal{F}$ to $\mathcal{L}$, where these spaces are understood to be defined with the appropriate parameters, and equipped with the metrics $d_{\mathcal{F}}$ and $d_\LL$. We will prove that this function is H\"older continuous, and then Theorem \ref{thm:HolDep-intro} will follow from Theorem \ref{thm:continuity}.

We will need the following lemma, which gives control over the distance between preimages of a point $x\in X_{n+1}$ under $T_n$ and $T_n'$, and is reminiscent of the argument in Lemma \ref{lem:Q-kappa}, where we paired preimages of two nearby points under a single map. 
Given two maps $T_n,T_n' \colon X_n \to X_{n+1}$, it will be convenient to use the following notation for the uniform distance between them:
\[
d_\infty(T_n,T_n') := \sup_{x\in X_n} d_{n+1}(T_n x, T_n' x).
\]

\begin{lemma}
\label{lem:preimages}
Suppose that $\bT$ and $\bT'$ satisfy \ref{A:pre}--\ref{A:exp} for the same choice of $D,\delta,\rho$, and that $n\in \ZZ$ satisfies $d_\infty(T_n,T_n') < \delta$. Then for every $x\in X_{n+1}$, there exists a bijection $\pi_x \colon T_n^{-1}(x) \to (T_n')^{-1}(x)$ such that
\begin{equation}\label{eqn:dpi}
d_n(y,\pi_x(y)) \leq d_\infty(T_n,T_n')
\quad\text{for all }
y\in T_n^{-1}(x).
\end{equation}
\end{lemma}
\begin{proof}
Given $x\in X_{n+1}$ and $y\in T_n^{-1}(x)$, we have
\[
d_{n+1}(T_n'(y), x)
= d_{n+1}(T_n'(y), T_n(y)) \leq d_\infty(T_n,T_n') < \delta,
\]
so $x\in B(T_n'(y),\delta) \subset T_n'(B(y,\delta))$ by \ref{A:exp}. Thus there exists a unique point $\pi_x(y) \in B(y,\delta)$ such that $T_n'(\pi_x(y)) = x$. Reversing the roles of $T_n$ and $T_n'$ produces an inverse map for $\pi_x$, so we see that $\pi_x$ is a bijection with the desired properties.
\end{proof}

We can now prove Theorem \ref{thm:HolDep-intro}.
For convenience, we write $L_nf = \bL(\bT,\bph)_n f$.

\begin{proof}[Proof of Theorem \ref{thm:HolDep-intro}]
As discussed above, it suffices to show that there exists $C>0$ such that whenever $\dF((\bT,\bph),(\bT',\bph'))$ is small, we have
    \[
    d_\LL(\bL,\bL')\leq Cd_{\mathcal{F}} ((\bT,\bph),(\bT',\bph'))^\beta.
    \]
(Here $\beta \in (0,1)$ is the same H\"older exponent as we assume for the potentials.)
This requires us to bound $\|L_n f - L_n' f\|_\infty$ for all $f\in \Lambda_n$ with $\|f\|_\infty=1$. Start with the observation that \ref{A:pre} and \ref{A:Hol} give $\|L_n f\|_\infty \leq D e^V \|f\|_\infty$, so for every $n\in \ZZ$, we have
\begin{equation}\label{eqn:LnDeV}
\|L_n - L_n' \|_{\Lambda_n} \leq 2D e^V.
\end{equation}

Let $N\in \NN$ be maximal such that $2^N d_{\mathcal{F}}((\bT,\bph),(\bT',\bph')) < \delta$. 
For $|n|>N$, we will use the trivial bound \eqref{eqn:LnDeV}. For $|n|\leq N$, we will see that we can do better using the fact that
\begin{equation}\label{eqn:dinf-dF}
d_\infty(T_n,T_n') \leq 2^{|n|} \dF((\bT,\bph),(\bT',\bph')) < \delta,
\end{equation}
which allows us to apply
Lemma \ref{lem:preimages}. In the following estimates, the supremum is over all pairs $(f,x) \in \Lambda_n(Q) \times X_{n+1}$ such that $\|f\|_\infty = 1$, and we write $y' := \pi_x(y)$  whenever $y\in T_n^{-1}(x)$:
\begin{equation}
    \label{eqn:TransOpBound}
    \begin{aligned}
        \|L_n-L'_n\|_{\Lambda_n} &\leq\sup_{(f,x)}\bigg|\sum_{y\in T_n^{-1}x}e^{\ph_n(y)}f(y)-\sum_{y'\in (T')_n^{-1}x}e^{\ph_n'(y')}f(y')\bigg|\\
        &\leq \sup_{(f,x)} \sum_{y\in T_n^{-1}x}\big(e^{\|\ph_n\|}\big|f({y}) - f(y')\big|
        + \|f\|_{\infty}\big|e^{\ph_n(y)}-e^{\ph_n'(y')}\big|\big).
    \end{aligned}
    \end{equation}

To bound the first term, we observe that the quantity $\frac 1t (e^{Qt} - 1)$ is bounded above on $(0,\delta^\beta]$, so there exists $C_4>0$ such that
\[
e^{Qt} - 1 \leq C_4t
\quad\text{for all } t\in (0,\delta^\beta].
\]
Since $|f(y) - f(y')| < \delta$ for each $y\in T_n^{-1} x$ by Lemma \ref{lem:preimages}, we have (using the fact that $f\in \Lambda_n(Q)$ and $\|f\|_\infty = 1$):
\begin{equation}\label{eqn:fyfy'}
|f(y) - f(y')| \leq e^{Q d_n(y,y')^\beta} - 1
\leq C_4 d_n(y,y')^\beta.
\end{equation}
To bound the second part of \eqref{eqn:TransOpBound}, observe that there is $C_5>0$ such that
\[
|e^t - e^s| \leq C_5 |t-s|
\quad\text{for all } t,s\in [-V,V],
\]
and thus
\begin{align*}
|e^{\ph_n(y)} - e^{\ph_n'(y')}| &\leq
C_5 |\ph_n(y) - \ph_n'(y')| \\
&\leq C_5 \big( |\ph_n(y) - \ph_n(y')| + |\ph_n(y') - \ph_n'(y')|\big) \\
&\leq C_5 \big( H d_n(y,y')^\beta + \|\ph_n - \ph_n'\|_\infty \big).
\end{align*}

Combining this bound with \eqref{eqn:TransOpBound} and \eqref{eqn:fyfy'}, we get
\begin{align*}
\|L_n-L'_n\|_{\Lambda_n}
&\leq \sup_{(f,x)} \sum_{y\in T_n^{-1} x}
\big( e^V C_4 d_n(y,y')^\beta
+ C_5 H d_n(y,y')^\beta + C_5 \|\ph_n - \ph_n'\|_\infty \big) \\
&\leq D(e^V C_4 + C_5 H) d_\infty(T_n,T_n')^\beta
+ DC_5 \|\ph_n - \ph_n'\|_\infty.
\end{align*}
Writing $C_6 := D(e^V C_4 + C_5 H)$, we can use this bound for $|n|\leq N$ together with the trivial bound \eqref{eqn:LnDeV} for $|n|> N$ to obtain
\begin{align*}
d_\LL(\bL,\bL')
&= \sum_{|n| > N} 2^{-|n|} \|L_n - L_n'\|_{\Lambda_n}
+ \sum_{|n| \leq N} 2^{-|n|} \|L_n - L_n'\|_{\Lambda_n} \\
&\leq \sum_{|n| > N} 2^{-|n|} 2D e^V
+ \sum_{|n| \leq N} 2^{-|n|} \big( C_6 d_\infty(T_n,T_n')^\beta + DC_5 \|\ph_n - \ph_n'\|_\infty \big).
\end{align*}
For convenience, write $R := \dF((\bT,\bph),(\bT',\bph'))$, and observe that
\[
d_\infty(T_n,T_n') \leq 2^{|n|} R,
\qquad\text{so}\qquad
d_\infty(T_n,T_n')^\beta \leq 2^{\beta|n|} R^\beta,
\]
which allows us to continue the above estimates by writing
\begin{align*}
d_\LL(\bL,\bL') &\leq 4D e^V 2^{-N}
+ \sum_{|n| \leq N} 2^{-|n|} C_6 2^{\beta|n|} R^\beta 
+ \sum_{|n| \leq N} 2^{-|n|} DC_5\|\ph_n - \ph_n'\|_\infty \\
&\leq 4De^V 2^{-N} + \Big(\sum_{n\in \ZZ} 2^{-(1-\beta)|n|} \Big) C_6 R^\beta
+ DC_5 \Big( \sum_{n\in \ZZ} 2^{-|n|} \|\ph_n - \ph_n'\|_\infty \Big).
\end{align*}
Since $N$ was maximal such that $2^N R < \delta$, we have $2^N R \geq \delta/2$, so $2^{-N} \leq 2\delta^{-1} R$, and writing $C_7 :=  \Big(\sum_{n\in \ZZ} 2^{-(1-\beta)|n|} \Big) C_6 < \infty$ (since $\beta \in (0,1)$), we obtain
\[
d_\LL(\bL,\bL') \leq 4De^V \cdot 2\delta^{-1} R
+ C_7 R^\beta + DC_5 R
\leq (8De^V\delta^{-1} + C_7 + DC_5) R^\beta,
\]
which completes the proof.
\end{proof}

\bibliographystyle{amsalpha}
\bibliography{Cones}
\end{document}